\documentclass{article}

\title{A Note on the Uniform Kan Condition in Nominal Cubical Sets}
\author{Robert Harper and Kuen-Bang Hou (Favonia)}
\date{}

\usepackage[T1]{fontenc}
\usepackage[mathbf]{euler}
\usepackage{amsmath,amssymb,amsthm}
\usepackage{manfnt} % for the cube
\usepackage{mathtools}
\usepackage{etoolbox}
\usepackage{natbib}
\usepackage[color=green!40]{todonotes}
\usepackage{verbatim}
\usepackage{booktabs}
\usepackage{tikz-cd}
\usepackage{soul}
\tikzset{new edge/.style={semithick}}
\tikzset{old edge/.style={dotted}}
\tikzset{hint/.style={dotted}}

\newrobustcmd*{\textbsf}[1]{\textbf{\textsf{#1}}}
\newrobustcmd*{\mathbsf}[1]{\mathbf{\mathsf{#1}}}
\newrobustcmd*{\mathoper}[1]{\mathop{#1}\nolimits}

\newrobustcmd*{\parens}[1]{(#1)}
\newrobustcmd*{\noparens}[1]{#1}
\newrobustcmd*{\sqbracks}[1]{[#1]}
\newrobustcmd*{\cssqbracks}[1]{\cs{[}#1\cs{]}}

\newrobustcmd*{\iso}{\cong}
\newrobustcmd*{\eqdef}{\mathrel{\triangleq}}

\newrobustcmd*{\dom}[1]{\mathoper{\textsf{dom}}\parens{#1}}
\newrobustcmd*{\cod}[1]{\mathoper{\textsf{cod}}\parens{#1}}
\renewrobustcmd*{\hom}[3]{\mathoper{\textsf{hom}}_{#1}\parens{#2,#3}}

\newrobustcmd*{\CC}{\mathord{\mbox{\manimpossiblecube}}}
\newrobustcmd*{\two}{\mathbf{2}}
\newrobustcmd*{\zero}{\mathbf{0}}
\newrobustcmd*{\one}{\mathbf{1}}
\newrobustcmd*{\ext}[2]{{#1},{#2}}
\newrobustcmd*{\minus}[2]{{#1}\setminus{#2}}
\newrobustcmd*{\comb}[2]{{#1},{#2}}
\newrobustcmd*{\bdcomb}[3]{\comb{\ext{#1}{#3}}{#2}}
\newrobustcmd*{\fset}[1]{\{#1\}}
\newrobustcmd*{\emp}{\emptyset}
\newrobustcmd*{\thesetst}[2]{\{\,#1\,\mid\,#2\,\}}
\newrobustcmd*{\famset}[2]{\{\,#1\,\}_{#2}}
\newrobustcmd*{\union}[2]{{#1}\mathbin{\cup}{#2}}

\newrobustcmd*{\id}{\mathord{\textsf{id}}}
\newrobustcmd*{\dcomp}[2]{{#1}\cdot{#2}}
\newrobustcmd*{\ocomp}[2]{{#2}\circ{#1}}
\newrobustcmd*{\inv}[1]{{#1}^{-1}}
\newrobustcmd*{\maps}[2]{{#1}\mapsto{#2}}
\newrobustcmd*{\extf}[3]{{#1}\sqbracks{\maps{x}{y}}}

\newrobustcmd*{\swap}[2]{{#1}\mathbin{\leftrightarrow}{#2}}
\newrobustcmd*{\spec}[2]{{#1}\mathbin{\mapsto}{#2}}
\newrobustcmd*{\incl}[1]{\iota_{#1}}
\newrobustcmd*{\contr}[3]{\maps{#1,#2}{#3}}

\newrobustcmd*{\coyoneda}[1]{\mathoper{\mathbf{y}}^{#1}}
\newrobustcmd*{\hhp}{\mathoper{\underline{\mathbf{h}}}}
\newrobustcmd*{\hhn}{\mathoper{\overline{\mathbf{h}}}}
\newrobustcmd*{\hhpb}[3]{\hhp^{\comb{#1}{#2}}_{#3}}
\newrobustcmd*{\hhnb}[3]{\hhn^{\comb{#1}{#2}}_{#3}}
\newrobustcmd*{\fscube}[1]{\mathoper{\mathbf{\square}}^{#1}}
\newrobustcmd*{\fspbox}[3]{\mathoper{\mathbf{\sqcup}}^{{#1}{;}{#2}}_{#3}}
\newrobustcmd*{\fsnbox}[3]{\mathoper{\mathbf{\sqcap}}^{{#1}{;}{#2}}_{#3}}
\newrobustcmd*{\agemO}{\text{\raisebox{-0.2ex}{\rotatebox[origin=c]{180}{$\Omega$}}}}
\newrobustcmd*{\pafm}[2]{\mathoper{\agemO}^{{#1}}_{#2}}
\newrobustcmd*{\nafm}[2]{\mathoper{\Omega}^{{#1}}_{#2}}
\newrobustcmd*{\pfill}[4]{\underline{\mathoper{\mathsf{fill}}}^{{#2}{;}{#3}}_{#4}}   %\cs{#1}{;}
\newrobustcmd*{\nfill}[4]{\overline{\mathoper{\mathsf{fill}}}^{{#2}{;}{#3}}_{#4}}     %\cs{#1}{;}
\newrobustcmd*{\plift}[4]{\underline{\mathoper{\mathsf{lift}}}^{{#2}{;}{#3}}_{#4}}  %\cs{#1}{;}
\newrobustcmd*{\nlift}[4]{\overline{\mathoper{\mathsf{lift}}}^{{#2}{;}{#3}}_{#4}}    %\cs{#1}{;}
\newrobustcmd*{\paproj}[3]{\underline{\mathoper{\mathsf{proj}}}^{#1;#2}_#3}
\newrobustcmd*{\naproj}[3]{\overline{\mathoper{\mathsf{proj}}}^{#1;#2}_#3}
\newrobustcmd*{\pnerve}[4]{\underline{\mathoper{\mathsf{nrv}}}^{\cs{#1}{;}{#2}{;}{#3}}_{#4}}
\newrobustcmd*{\nnerve}[4]{\overline{\mathoper{\mathsf{nrv}}}^{\cs{#1}{;}{#2}{;}{#3}}_{#4}}
\newrobustcmd*{\prealize}[4]{\underline{\mathoper{\mathsf{rlz}}}^{\cs{#1}{;}{#2}{;}{#3}}_{#4}}
\newrobustcmd*{\nrealize}[4]{\overline{\mathoper{\mathsf{rlz}}}^{\cs{#1}{;}{#2}{;}{#3}}_{#4}}

\newrobustcmd*{\algpbox}[4]{\mathoper{\mathbf{\underline{#1}}}^{{#2}{;}{#3}}_{#4}}
\newrobustcmd*{\algnbox}[4]{\mathoper{\mathbf{\overline{#1}}}^{{#2}{;}{#3}}_{#4}}

\newrobustcmd*{\geopbox}[4]{\mathoper{\mathsf{\underline{geobox}}}\sqbracks{\cfib{#1}}^{{#2}{;}{#3}}_{#4}}
\newrobustcmd*{\geonbox}[4]{\mathoper{\mathsf{\overline{geobox}}}\sqbracks{\cfib{#1}}^{{#2}{;}{#3}}_{#4}}

\newrobustcmd*{\SET}{\mathord{\textbsf{Set}}}
\newrobustcmd*{\unitset}{1}
\newrobustcmd*{\prodset}[2]{{#1}\times{#2}}
\newrobustcmd*{\funset}[2]{{#1}\to{#2}}
\newrobustcmd*{\voidset}{\emptyset}
\newrobustcmd*{\sumset}[2]{{#1}+{#2}}

\newrobustcmd*{\cs}[1]{\boldsymbol{#1}}
\newrobustcmd*{\cfib}[1]{{#1}} % cubical fibration
\newrobustcmd*{\alg}[1]{\boldsymbol{#1}}
\newrobustcmd*{\bd}{\mathoper{\partial}}
\newrobustcmd*{\face}[3]{\mathop{\bd_{#1}^{#2}}\parens{#3}}

\newrobustcmd*{\catpath}[2]{{#1}\cdot{#2}}
\newrobustcmd*{\nullpath}[1]{\varepsilon_{#1}}
\newrobustcmd*{\invpath}[1]{{#1}^{-1}}
\newrobustcmd*{\fwdpath}[1]{\overrightarrow{#1}}
\newrobustcmd*{\revpath}[1]{\overleftarrow{#1}}

\newrobustcmd*{\CSET}{\mathord{\textbsf{cSet}}}

\newrobustcmd*{\coelt}[2]{\int^{#1}{#2}}

\newrobustcmd*{\IsCtx}[1]{\vdash{#1}}
\newrobustcmd*{\IsSub}[3]{{#2}\vdash{#1}:{#3}}
\newrobustcmd*{\IsTyp}[2]{{#1}\vdash{#2}}
\newrobustcmd*{\IsElt}[3]{{#1}\vdash{#2}:{#3}}

\newrobustcmd*{\empctx}{\bullet}
\newrobustcmd*{\csempctx}{\unitset}
\newrobustcmd*{\extctx}[2]{{#1}\cdot{#2}}
\newrobustcmd*{\csextctx}[2]{\extctx{#1}{#2}}
\newrobustcmd*{\itelt}{\textsf{it}}
\newrobustcmd*{\csitelt}{\textbsf{it}}

\newrobustcmd*{\empsub}{\bullet}
\newrobustcmd*{\csempsub}{\cs{\empsub}}
\newrobustcmd*{\prevsub}{\textsf{prev}}
\newrobustcmd*{\csprevsub}{\textbsf{prev}}
\newrobustcmd*{\extsub}[2]{{#1}\cdot {#2}}
\newrobustcmd*{\csextsub}[2]{\extsub{#1}{#2}}
\newrobustcmd*{\compsub}[2]{\dcomp{#1}{#2}}
\newrobustcmd*{\cscompsub}[2]{\compsub{#1}{#2}}

\newrobustcmd*{\inst}[2]{{#1}\sqbracks{#2}}
\newrobustcmd*{\csinst}[2]{{#1}\cssqbracks{#2}}

\newrobustcmd*{\unittyp}{\textsf{unit}}
\newrobustcmd*{\csunittyp}{\textbsf{unit}}
\newrobustcmd*{\unitelt}{\langle\rangle}
\newrobustcmd*{\csunitelt}{\cs{\langle}\cs{\rangle}}

\newrobustcmd*{\prodtyp}[2]{{#1}\times{#2}}
\newrobustcmd*{\csprodtyp}[2]{{#1}\mathbin{\cs{\times}}{#2}}
\newrobustcmd*{\pairelt}[2]{\langle{#1},{#2}\rangle}
\newrobustcmd*{\cspairelt}[2]{\cs{\langle}{#1},{#2}\cs{\rangle}}
\newrobustcmd*{\fstelt}[1]{\mathoper{\textsf{fst}}\parens{#1}}
\newrobustcmd*{\csfstelt}[1]{\mathoper{\textbsf{fst}}\parens{#1}}
\newrobustcmd*{\sndelt}[1]{\mathoper{\textsf{snd}}\parens{#1}}
\newrobustcmd*{\cssndelt}[1]{\mathoper{\textbsf{snd}}\parens{#1}}

\newrobustcmd*{\funtyp}[2]{{#1}\to{#2}}
\newrobustcmd*{\csfuntyp}[2]{{#1}\mathbin{\cs{\to}}{#2}}
\newrobustcmd*{\lamelt}[2]{\lambda_{#1}\parens{#2}}
\newrobustcmd*{\cslamelt}[2]{\cs{\lambda}_{#1}\parens{#2}}
\newrobustcmd*{\appelt}[2]{{#1}\parens{#2}}
\newrobustcmd*{\csappelt}[2]{\appelt{#1}{#2}}

\newrobustcmd*{\sigtyp}[3]{\Sigma {#2}{:}{#1}{.}{#3}}
\newrobustcmd*{\cssigtyp}[3]{\cs{\Sigma}{#2}\cs{:}{#1}\cs{.}{#3}}
\newrobustcmd*{\pityp}[3]{\Pi {#2}{:}{#1}{.}{#3}}
\newrobustcmd*{\cspityp}[3]{\cs{\Pi}{#2}\cs{:}{#1}\cs{.}{#3}}

\newrobustcmd*{\voidtyp}{\textsf{void}}
\newrobustcmd*{\csvoidtyp}{\textbsf{void}}
\newrobustcmd*{\sumtyp}[2]{{#1}+{#2}}
\newrobustcmd*{\cssumtyp}[2]{{#1}\mathbin{\cs{+}}{#2}}

\newrobustcmd*{\idtyp}[3]{\mathoper{\textsf{Id}}_{#1}\parens{{#2},{#3}}}
\newrobustcmd*{\csidtyp}[3]{\mathoper{\textbsf{Id}}_{#1}\parens{{#2},{#3}}}
\newrobustcmd*{\reflelt}[2]{\mathoper{\textsf{refl}}_{#1}\parens{#2}}
\newrobustcmd*{\csreflelt}[2]{\mathoper{\textbsf{refl}}_{#1}\parens{#2}}
\newrobustcmd*{\csreflcube}[1]{\csreflelt{}{#1}}
\newrobustcmd*{\jelt}[2]{\mathoper{\textsf{J}}\parens{#1,#2}}
\newrobustcmd*{\csjelt}[2]{\mathoper{\textbsf{J}}\parens{#1,#2}}

\renewrobustcmd*{\fam}[2]{{#1}.{#2}}
\newrobustcmd*{\tr}[3]{\mathoper{\textsf{tr}}\sqbracks{#1}\parens{#2}\parens{#3}}

\newrobustcmd*{\idover}[4]{{#3}=^{#1}_{#2}{#4}}

\newrobustcmd*{\absdir}[2]{{#1}\cs{.}{#2}}

\newrobustcmd*{\ivtyp}{\textsf{I}}
\newrobustcmd*{\csivtyp}{\textbsf{I}}
\newrobustcmd*{\ivzelt}{\textsf{0}}
\newrobustcmd*{\csivzelt}{\textbsf{0}}
\newrobustcmd*{\ivoelt}{\textsf{1}}
\newrobustcmd*{\csivoelt}{\textbsf{1}}
\newrobustcmd*{\ivsegelt}{\textsf{seg}}
\newrobustcmd*{\csivsegelt}{\textbsf{seg}}

\newrobustcmd*{\cscirctyp}{\mathbsf{S^1}}
\newrobustcmd*{\csbaseelt}{\mathbsf{\ast}}
\newrobustcmd*{\csloopelt}{\textbsf{loop}}

\newrobustcmd*{\fib}[1]{\mathoper{\mathsf{fib}}\sqbracks{#1}}
\newrobustcmd*{\trans}[4]{\mathoper{\mathsf{trans}}\sqbracks{#1}^{#2}_{#3}\parens{#4}}

\newrobustcmd*{\spantuple}[2]{({#1},{#2})}
\newrobustcmd*{\cospandiagram}[5]{{#1}\xrightarrow{#4}{#3}\xleftarrow{#5}{#2}}
\newrobustcmd*{\spandiagram}[5]{{#1}\xleftarrow{#4}{#3}\xrightarrow{#5}{#2}}
\newrobustcmd*{\tuple}[2]{\langle{#1}\mathbin{;}{#2}\rangle}
\newrobustcmd*{\case}[2]{\{{#1}\mathbin{;}{#2}\}}

\newtheorem{proposition}{Proposition}

\begin{document}
\maketitle{}

\begin{abstract}
  Bezem, Coquand, and Huber have recently given a constructively valid model of higher type theory in a category of
  nominal cubical sets satisfying a novel condition, called the \emph{uniform Kan condition (UKC)}, which generalizes
  the standard cubical Kan condition (as considered by, for example, Williamson in his survey of combinatorial homotopy
  theory) to admit phantom ``additional'' dimensions in open boxes.  This note, which represents the authors'
  attempts to fill in the details of the UKC, is intended for newcomers to the field who may appreciate a more explicit
  formulation and development of the main ideas.  The crux of the exposition is an analogue of the Yoneda Lemma for
  co-sieves that relates geometric open boxes bijectively to their algebraic counterparts, much as its progenitor for
  representables relates geometric cubes to their algebraic counterparts in a cubical set.  This characterization is
  used to give a formulation of uniform Kan fibrations in which uniformity emerges as naturality in the additional
  dimensions.
\end{abstract}

\section{Cubical Sets}
\label{sec:cs}

In their recent landmark paper \cite{bch} present a novel formulation of cubical sets using symbols (or names) for the
\emph{dimensions} (or \emph{coordinates}) of an $n$-dimensional cube.\footnote{\cite{pitts:cubical} gives a formulation
  in terms of his category of nominal sets~\citep{pitts:nominal-sets}} Briefly, they define a category $\CC$ whose
objects are finite sets, $I$, of symbols and whose morphisms $f:I\to J$ are set functions $I\to J+\two$, where $\two$ is
the set $\fset{\zero,\one}$, that are injective on the preimage of $J$.  Identities are identities, and composition is
given as in the Kleisli category for a ``two errors'' monad---the symbols $\zero$ and $\one$ are propagated, and
otherwise the morphisms are composed as functions.

If an object $I$ of $\CC$ has $n$ elements, it is said to be an $n$-dimensional set of dimensions.  The notation
$\ext{I}{x}$ is defined for a symbol $x\notin I$ to be $I\cup\fset{x}$, and is extended to $\ext{I}{x,y,\dots}$ for a
finite sequence of distinct symbols not in $I$ in the evident way.  The notation $\minus{I}{x}$, where $x\in I$, is just
the set $I$ with $x$ omitted.  The morphism $\extf{f}{x}{y}:\ext{I}{x}\to \ext{J}{y}$, where $f:I\to J$, $x\notin I$,
and $y\notin J$, extends $f$ by sending $z\in I$ to $f(z)$ and $x$ to $y$.

Some special morphisms of $\CC$ are particularly important:
\begin{enumerate}
\item Identity: $\id:I\to I$ sending $x\in I$ to itself.
\item Composition: $\dcomp{f}{g}$, $\ocomp{f}{g}:I\to K$, in diagrammatic and conventional form, where $f:I\to J$ and
  $g:J\to K$, defined earlier.
\item Exchange: $\swap{x}{y}:\ext{I}{x,y}\to\ext{I}{x,y}$, where $x,y\notin I$, sending $x$ to $y$ and $y$ to $x$.
\item Specialize: $\spec{x}{\zero}$, $\spec{x}{\one}: \ext{I}{x}\to I$ sending $x$ to $\zero$ and $\one$, respectively.
\item Inclusion: $\incl{x}:I\to \ext{I}{x}$, where $x\notin I$, sending $y\in I$ to $y$.
\end{enumerate}
These are all ``polymorphic'' in the ambient sets $I$ of symbols, as indicated implicitly.  It is useful to keep
in mind that any $f:I\to J$ in the cube category can be written as a composition of specializations, followed by
permutations, followed by inclusions.

A \emph{cubical set} is a covariant presheaf (i.e., a co-preasheaf) on the cube category, which is a functor $\cs{X}:\CC\to\SET$.  Explicitly, this provides
\begin{enumerate}
\item For each object $I$ of $\CC$, a set $\cs{X}(I)$, or $\cs{X}_I$.
\item For each morphism $f:I\to J$ of $\CC$, a function $\cs{X}(f)$, or $\cs{X}_f$, in $\cs{X}_I\to\cs{X}_J$ respecting identity and composition.
\end{enumerate}
As a convenience write $\cs{X}_x$ or $\cs{X}\fset{x}$, for $\cs{X}_{\fset{x}}$, and $\cs{X}_{x,y}$ or
$\cs{X}\fset{x,y}$, for $\cs{X}_{\fset{x,y}}$, and so forth.

Think of $\cs{X}_I$ as consisting of the $I$-cubes, which are $n$-dimensional cubes presented using the $n$ dimensions
given by $I$.  This interpretation is justified by the structure induced by the distinguished morphisms in the cube
category:
\begin{enumerate}
\item $\cs{X}_{\spec{x}{\zero}}$ and $\cs{X}_{\spec{x}{\one}}$ mapping $X_{\ext{I}{x}}\to X_I$ are called \emph{face
    maps} that compute the two $(n-1)$-dimensional faces of an $n$-dimensional cube in $\cs{X}_I$ along dimension $x$,
  where $I$ is an $n$-dimensional set of dimensions.  By convention, in low dimensions, $\spec{x}{\zero}$ designates the
  ``left'' or ``bottom'' or ``front'' face, and $\spec{x}{\one}$ designates the ``right'' or ``top'' or ``back'' face,
  according to whether one visualizes the dimension $x$ as being horizontal or vertical or perpendicular.
\item $\cs{X}_{\incl{x}}:\cs{X}_I\to\cs{X}_{\ext{I}{x}}$ is called a \emph{degeneracy map} that treats an $n$-cube as a
  degenerate $(n+1)$-cube, with degeneracy in the dimension $x$.  So, for example, a line in the $x$ dimension may be
  regarded as a square in dimensions $x$ and $y$, corresponding to the reflexive identification of the line with itself.
  Similarly, a point may be thought of as a degenerate-in-$x$ line, and thence as a degenerate-in-$y$ square, or it may
  be thought of as a degenerate-in-$y$ line and then a degenerate-in-$x$ square.
  The equation $\ocomp{\incl{x}}{\incl{y}} = \ocomp{\incl{y}}{\incl{x}}$ along with functoriality
  guarantees that these two degenerate squares are the same;
  a point can thus be thought of as a degenerate-in-$x$-and-$y$ square ``directly''.
\item $\cs{X}_{\swap{x}{y}}:\cs{X}_{\ext{I}{x,y}}\to\cs{X}_{\ext{I}{x,y}}$ is a \emph{change of coordinates map} that
  swaps the names of the dimensions $x$ and $y$; it is necessarily a bijection.  A change of coordinates map is also
  regarded as a degeneracy map in~\cite{bch}
\end{enumerate}
The face maps justify thinking of $\cs{X}_{\emptyset}$ as the \emph{points} of $\cs{X}$, of $\cs{X}_{\fset{x}}$ as the
\emph{lines} (in dimension named $x$), of $\cs{X}_{\fset{x,y}}$ as the \emph{squares}, and so on.  Compositions of face
maps are again face maps, and so too for degeneracies and changes of coordinates.  When $I\subseteq J$, the generalized
inclusion $\incl{}:I\to J$ stands for the evident composition of inclusion maps, in any order, and if $\pi:I\to J$ is a
permutation of sets, then it may be regarded as a change of coordinates map by treating it as a composition of
exchanges.  When the cubical set $\cs{X}$ involved is clear from context, we sometimes write $\face{x}{i}{\kappa}$ for
$\cs{X}_{\spec{x}{i}}$.

It is well to remember that, being a subcategory of $\SET$, the maps in the cubical category enjoy the equational
properties of functions, and that such equations are necessarily respected by any cubical set.  Thus, for example, a
cube $\kappa\in\cs{X}_I$ determines a degenerate $\ext{I}{x}$ cube $\cs{X}_{\incl{x}}(\kappa)\in \cs{X}_{\ext{I}{x}}$ in
the sense that its end points are both $\kappa$:
\begin{align*}
  \cs{X}_{\spec{x}{0}}(\cs{X}_{\incl{x}}(\kappa)) & = \cs{X}_{\ocomp{\incl{x}}{\parens{\spec{x}{0}}}}(\kappa) \\
  & = \cs{X}_{\id}(\kappa) \\
  & = \kappa,
\end{align*}
and the same holds true for the specialization $\spec{x}{1}$.  These equations follow directly from the definition of
the involved morphisms in $\CC$ as certain functions on finite sets.  As a consequence of these \emph{cubical
  identities} every morphism in $\CC$ is equal to a composition of face maps followed by a composition of exchanges
followed by a composition of degeneracies.

\smallskip

A morphism $F:\cs{X}\to\cs{Y}$ of cubical sets is, of course, a natural transformation between them, as functors into
$\SET$.  That is, for each object $I$ of $\CC$ there is a function $F_I:\cs{X}_I\to\cs{Y}_I$ such that for each map
$f:I\to J$ in $\CC$ the equation $\ocomp{F_I}{\cs{Y}_f}=\ocomp{\cs{X}_f}{F_J}$.  Identities and compositions are defined
as usual.  Cubical sets thereby form a category, $\CSET$.

Let $\cs{X}$ be a cubical set, and let $\kappa\in\cs{X}_I$ for some object $I$ of $\CC$.  That is, $\kappa$ is an
$I$-cube of $\cs{X}$ of dimension $n$.  It is useful to consider the cubes $\cs{X}_f(\kappa)\in\cs{X}_J$ as
$f$ ranges over all maps $f:I\to J$ in $\CC$.  Such cubes may be considered as the exposition of the \emph{cubical structure} of
$\kappa$ in the sense that they determine these aspects of $\kappa$:
\begin{enumerate}
\item The $m$-dimensional faces of $\kappa$ for each $m<n$, expressed in terms of various choices of dimensions for
  dimension $m$.  These are determined by the specialization morphisms of $\CC$.
\item The re-orientations of $\kappa$, which are determined by the exchange morphisms of $\CC$.
\item The $m$-dimensional degeneracies of $\kappa$, for each $m>n$, expressed in terms of various choices of dimensions.
  These are given by the inclusion morphisms of $\CC$.
\end{enumerate}
This leads directly to the definition of the \emph{free-standing} $I$-cube, notated $\fscube{I}$,
as the co-representable cubical set
\begin{displaymath}
  \fscube{I} \eqdef \hom{}{I}{-} : \CC\to\SET.
\end{displaymath}
A morphism $\kappa:\fscube{I}\to\cs{X}$ from the free-standing $I$-cube into $\cs{X}$ may be seen as specifying an
$I$-cube in $\cs{X}$.  By the Yoneda Lemma there is a bijection
\begin{displaymath}
  \hom{\CSET}{\fscube{I}}{\cs{X}} \cong \cs{X}_{I},
\end{displaymath}
so that the two concepts of a cube in $\cs{X}$ coincide (up to a bijection that is not usually notated explicitly).

\section{Uniform Kan Complexes}
\label{sec:ukc}

A central idea of~\cite{bch} is the \emph{uniform Kan condition}, a generalization of the well-known Kan condition, that
ensures that a cubical set has sufficient structure to support the interpretation of cubes as identifications.  To
motivate the condition, let us first consider the \emph{homotopic interval}, which is supposed to be an abstraction of
the unit interval on the real line, consisting of two end points, $\ivzelt$ and $\ivoelt$, and a line, $\ivsegelt$,
between them.  The idea is that a mapping from the interval into a space $\cs{X}$ is a ``drawing'' of a line in the
space $\cs{X}$, and hence can be used to construct paths and mediate their homotopies.

It is not clear at once what should be the definition of the interval, but in a spirit of minimalism one
might consider it to be given by the following equations at dimensions $0$ and $1$, and to consist solely of the
requisite degeneracies at all other dimensions:%
\footnote{
  Equivalently, it is the free-standing cube of the singleton set $\fset{x}$ for some dimension $x$.}
\begin{align}
  \label{eq:iv-def-cubical}
  \csivtyp_\emp & \eqdef{} \fset{\ivzelt,\ivoelt} && \\
  \csivtyp_{\fset{x}} & \eqdef{} \fset{x^{\ivzelt},x^{\ivoelt},\ivsegelt} && \text{(any name $x$)}.
\end{align}
This is ``a'' definition of some cubical set, but is it the ``right'' definition?
That is, does it support the interpretation of cubes as identifications?
The presence of the line
$\ivsegelt$ immediately raises questions about the inverse of $\ivsegelt$, and the absence of evidence for the
$\omega$-groupoid laws.

A convenient, direct method for defining the interval is described%
\footnote{Note that we are not referring to the unit interval defined in~\cite{bch},
but the general construction of a cubical set from a groupoid defined later in the paper.}
in~\cite{bch}.  First, construct the interval as a strict
groupoid generated by $\ivzelt$, $\ivoelt$, and $\ivsegelt$, and second, apply a general theorem showing that every
strict groupoid may be turned into a cubical set supporting cubes as identifications.  The latter is accomplished by
essentially taking objects of the strict groupoid as points, morphisms in the strict groupoid as lines, and add squares
sufficient to ensure that the groupoid laws are all properly witnessed.  In the case of the interval the construction is
illustrated by the first three dimensions of the interval given in Figure~\ref{fig:iv-pts-lines-squares}.
After some calculation it turns out that
each $I$-cube in this constructed cubical set is an assignment of corners of an unlabeled $I$-cube
to $\fset{0,1}$.  The assignment \begin{tikzcd} \bullet_0 \arrow{r} & \bullet_1 \end{tikzcd}
corresponds to $\ivsegelt$ and the assignment \begin{tikzcd} \bullet_1 \arrow{r} & \bullet_0 \end{tikzcd} is its inverse.
There are $2^{|I|}$ corners of an unlabeled $I$-cube, and thus the number of $I$-cubes (assignments)
is therefore $2^{2^{|I|}}$.
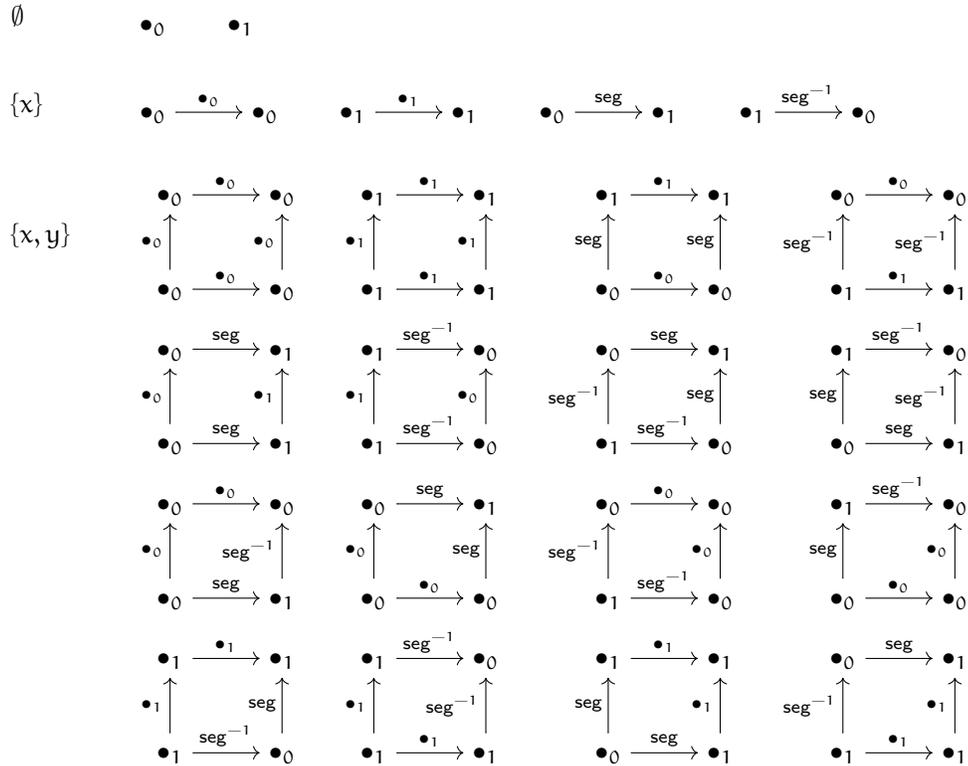
\begin{figure}[ht]
  \begin{displaymath}
    \begin{array}{l@{\qquad}l}
      \emptyset & 
      \begin{tikzcd}
        \bullet_0
      \end{tikzcd}
      \quad
      \begin{tikzcd}
        \bullet_1
      \end{tikzcd}
      \\[4ex]
      \fset{x} &
      \begin{tikzcd}
        \bullet_0 \arrow{r}{\bullet_0} & \bullet_0
      \end{tikzcd}
      \quad
      \begin{tikzcd}
        \bullet_1 \arrow{r}{\bullet_1} & \bullet_1
      \end{tikzcd}
      \quad
      \begin{tikzcd}
        \bullet_0 \arrow{r}{\ivsegelt} & \bullet_1
      \end{tikzcd}
      \quad
      \begin{tikzcd}
        \bullet_1 \arrow{r}{\invpath{\ivsegelt}} & \bullet_0
      \end{tikzcd}
      \\[4ex]
      \fset{x,y} &
      \begin{array}[t]{rrrr}
        \begin{tikzcd}
          \bullet_0
          \arrow{r}{\bullet_0} &
          \bullet_0 \\
          \bullet_0
          \arrow{u}{\bullet_0}
          \arrow{r}{\bullet_0} &
          \bullet_0
          \arrow{u}{\bullet_0}
        \end{tikzcd}
        &
        \begin{tikzcd}
          \bullet_1
          \arrow{r}{\bullet_1} &
          \bullet_1 \\
          \bullet_1
          \arrow{u}{\bullet_1}
          \arrow{r}{\bullet_1} &
          \bullet_1
          \arrow{u}{\bullet_1}
        \end{tikzcd}
        &
        \begin{tikzcd}
          \bullet_1
          \arrow{r}{\bullet_1} &
          \bullet_1 \\
          \bullet_0
          \arrow{u}{\ivsegelt}
          \arrow{r}{\bullet_0} &
          \bullet_0
          \arrow{u}{\ivsegelt}
        \end{tikzcd}
        &
        \begin{tikzcd}
          \bullet_0
          \arrow{r}{\bullet_0} &
          \bullet_0 \\
          \bullet_1
          \arrow{u}{\invpath{\ivsegelt}}
          \arrow{r}{\bullet_1} &
          \bullet_1
          \arrow{u}{\invpath{\ivsegelt}}
        \end{tikzcd}
        \\
        \begin{tikzcd}
          \bullet_0
          \arrow{r}{\ivsegelt} &
          \bullet_1 \\
          \bullet_0
          \arrow{u}{\bullet_0}
          \arrow{r}{\ivsegelt} &
          \bullet_1
          \arrow{u}{\bullet_1}
        \end{tikzcd}
        &
        \begin{tikzcd}
          \bullet_1
          \arrow{r}{\invpath{\ivsegelt}} &
          \bullet_0 \\
          \bullet_1
          \arrow{u}{\bullet_1}
          \arrow{r}{\invpath{\ivsegelt}} &
          \bullet_0
          \arrow{u}{\bullet_0}
        \end{tikzcd}
        &
        \begin{tikzcd}
          \bullet_0
          \arrow{r}{\ivsegelt} &
          \bullet_1 \\
          \bullet_1
          \arrow{u}{\invpath{\ivsegelt}}
          \arrow{r}{\invpath{\ivsegelt}} &
          \bullet_0
          \arrow{u}{\ivsegelt}
        \end{tikzcd}
        &
        \begin{tikzcd}
          \bullet_1
          \arrow{r}{\invpath{\ivsegelt}} &
          \bullet_0 \\
          \bullet_0
          \arrow{u}{\ivsegelt}
          \arrow{r}{\ivsegelt} &
          \bullet_1
          \arrow{u}{\invpath{\ivsegelt}}
        \end{tikzcd}
        \\
        \begin{tikzcd}
          \bullet_0
          \arrow{r}{\bullet_0} &
          \bullet_0 \\
          \bullet_0
          \arrow{u}{\bullet_0}
          \arrow{r}{\ivsegelt} &
          \bullet_1
          \arrow{u}{\invpath{\ivsegelt}}
        \end{tikzcd}
        &
        \begin{tikzcd}
          \bullet_0
          \arrow{r}{\ivsegelt} &
          \bullet_1 \\
          \bullet_0
          \arrow{u}{\bullet_0}
          \arrow{r}{\bullet_0} &
          \bullet_0
          \arrow{u}{\ivsegelt}
        \end{tikzcd}
        &
        \begin{tikzcd}
          \bullet_0
          \arrow{r}{\bullet_0} &
          \bullet_0 \\
          \bullet_1
          \arrow{u}{\invpath{\ivsegelt}}
          \arrow{r}{\invpath{\ivsegelt}} &
          \bullet_0
          \arrow{u}{\bullet_0}
        \end{tikzcd}
        &
        \begin{tikzcd}
          \bullet_1
          \arrow{r}{\invpath{\ivsegelt}} &
          \bullet_0 \\
          \bullet_0
          \arrow{u}{\ivsegelt}
          \arrow{r}{\bullet_0} &
          \bullet_0
          \arrow{u}{\bullet_0}
        \end{tikzcd}
        \\
        \begin{tikzcd}
          \bullet_1
          \arrow{r}{\bullet_1} &
          \bullet_1 \\
          \bullet_1
          \arrow{u}{\bullet_1}
          \arrow{r}{\invpath{\ivsegelt}} &
          \bullet_0
          \arrow{u}{\ivsegelt}
        \end{tikzcd}
        &
        \begin{tikzcd}
          \bullet_1
          \arrow{r}{\invpath{\ivsegelt}} &
          \bullet_0 \\
          \bullet_1
          \arrow{u}{\bullet_1}
          \arrow{r}{\bullet_1} &
          \bullet_1
          \arrow{u}{\invpath{\ivsegelt}}
        \end{tikzcd}
        &
        \begin{tikzcd}
          \bullet_1
          \arrow{r}{\bullet_1} &
          \bullet_1 \\
          \bullet_0
          \arrow{u}{\ivsegelt}
          \arrow{r}{\ivsegelt} &
          \bullet_1
          \arrow{u}{\bullet_1}
        \end{tikzcd}
        &
        \begin{tikzcd}
          \bullet_0
          \arrow{r}{\ivsegelt} &
          \bullet_1 \\
          \bullet_1
          \arrow{u}{\invpath{\ivsegelt}}
          \arrow{r}{\bullet_1} &
          \bullet_1
          \arrow{u}{\bullet_1}
        \end{tikzcd}
      \end{array}
    \end{array}
  \end{displaymath}
  \caption{Cubes of the First Three Dimensions of $\ivtyp$}
  \label{fig:iv-pts-lines-squares}
\end{figure}

Another possible definition of the interval is by a glueing construction that simply specifies that the two points $0$
and $1$ are to be identified.  Such glueing constructions are usually presented by a pushout construction, which exactly
expresses the identification of specified elements of two disjoint sets.  The pushout method is essentially a special
case of the concept of a \emph{higher inductive definition}~\citep{hott-book}, which allows identifications at all
dimensions to be specified among the elements generated by given points.  For example, the interval is the free
$\omega$-groupoid generated by $\ivzelt$, $\ivoelt$, and $\ivsegelt$~\citep{hott-book}.  The universal property states
that to define a map out of the interval into another type, it suffices to specify its behavior on the points $\ivzelt$
and $\ivoelt$ in such a way that the identification $\ivsegelt$ is sent to an identification of the images of the end
points.  It can be shown that, in the presence of type constructors yet to be formulated in the higher-dimensional case,
this formulation suffices to ensure that the necessary identifications are generated by a higher inductive definition.

The discussion of the interval is intended to raise the question: when is a cubical set a \emph{type}?  One answer is
simple enough, but hardly informative: exactly when it is possible to interpret the rules of type theory into it.  But
it would be pleasing to find some syntax-free criteria for specifying when cubes may be correctly understood as
heterogeneous identifications of their faces in a type.  So far we have argued, rather informally, that the least to be
expected is that identifications be catenable, when compatible, and reversible, along a given dimension, and that the
corresponding groupoid laws hold, at least up to higher identifications.  These conditions are surely
necessary, but are they sufficient?  Another criterion, whose importance will emerge shortly, is \emph{functoriality},
which states that families of types and families of terms should respect identifications of their free variables in an
appropriate sense.  In the case of families of types their interpretation as cubical families of sets implies that
isomorphic sets should be assigned to identified indices.  Finally, another condition, which is not possible to explain
fully just now, will also be required to ensure that the elimination form for the identification type behaves properly
even in the presence of non-trivial higher-dimensional structure.

One may consider that these should be the defining conditions for a cubical set to be a type, but they are distressingly
close to the answer based on the rules of type theory.  \citeauthor{bch} propose that a type is a \emph{uniform Kan
  complex} and that a family of types is a \emph{uniform Kan fibration}.%
\footnote{More precisely, \citeauthor{bch} interpret a family of types into a uniform Kan cubical family of sets, not a
  uniform Kan fibration.}  The formulation of these conditions goes back to pioneering work of Daniel Kan in the 1950's
on the question of when is a cubical set a suitable basis for doing homotopy theory?  Kan gave an elegant condition that
ensures, in one simple criterion, that cubes behave like identifications.  The UKC, introduced by~\citeauthor{bch},
generalizes the Kan condition and enables a constructively valid formulation.  Unfortunately, it is not known whether
the uniform Kan criteria are sufficient to support the interpretation of standard type theory.  In particular, the
definitional equivalence required of the elimination form for the identification type has not been validated by the
model---it holds only up to higher identification.  Still, the UKC comes very close to providing a semantic criterion
for a cubical set to be a type.  The remainder the paper is based on the UKC criterion, which we now describe in greater
detail.

\smallskip  
 
There are two formulations of the Kan condition for cubical sets, one more \emph{geometric} in flavor, the other more
\emph{algebraic}~\citep{williamson2012combinatorial}.  The algebraic formulation is more suitable for implementation
(giving constructive meaning to the concept of a Kan complex), whereas the geometric formulation is historically prior,
and more natural from the perspective of homotopy theory.  What follows is a development of both formulations, and their
relationship to one another.  Specifically, the \emph{homotopy lifting property} proposed by Kan corresponds to the
\emph{uniform box-filling operation} proposed by~\citeauthor{bch}.  The following account differs considerably from that
given by~\citeauthor{bch} in terms of the technical development, but the conceptual foundations remain the same.

Recall that the free-standing $I$-cube, $\fscube{I}$, is defined to be the co-representable cubical set
$\coyoneda{I}=\hom{}{I}{-}:\CC\to\SET$.  Thus,
\begin{align}
  \label{def:fscube}
  \fscube{I}_J & \eqdef \hom{}{I}{J} && (J:\CC) \\
  \fscube{I}_f(g) & \eqdef \ocomp{g}{f} && (f:J\to J',g:I\to J)
\end{align}
A \emph{geometric}, or \emph{external}, $I$-cube in $\cs{X}$ is a morphism $\kappa:\fscube{I}\to\cs{X}$ from the
free-standing $I$-cube into $\cs{X}$.  By the Yoneda Lemma each external $I$-cube, $\kappa$, of $\cs{X}$ determines an
\emph{algebraic}, or \emph{internal}, $I$-cube $\alg{\kappa}\in\cs{X}_I$ given by $\kappa_I\parens{\id}$, and every internal
$I$-cube of $\cs{X}$ arises in this way.  Thus $\cs{X}_I$ may be considered to \emph{represent} all and only the
geometric $I$-cubes of $\cs{X}$.

A \emph{free-standing box} is a cubical set given by a subfunctor of the free-standing cube (that is, a co-sieve)
determined by four parameters:
\begin{enumerate}
\item A set, $I$, specifying the \emph{included} dimensions of the box.
\item A set, $J$, specifying the \emph{extra} dimensions of the box.
\item A dimension, $y\notin I\cup J$, specifying the \emph{filling dimension} of the box.
\item The \emph{polarity} of the box, which specifies the \emph{filling direction}.
\end{enumerate}
A free-standing box is required to have both faces in its included dimensions, and one face in its filling dimension,
the \emph{starting face}, so that it always has an odd number of faces (as few as $1$ when $I=\emp$ and as many as
$2n+1$ when $J=\emp$ and $I$ is of size $n$).  In the standard formulation of cubical sets, such as the one given
by~\cite{williamson2012combinatorial}, there are no extra faces, so that a box may be viewed as a cube with its
interior and one face omitted, namely the opposite face to the starting face in the filling dimension, which is called
the \emph{ending face}, or \emph{composition}, of the filler.  The more general form of box considered given here that
of~\citeauthor{bch}, and gives rise to the need to consider the relationship between the filler of certain faces of a
box and the filler of the whole box---their \emph{uniformity condition} ensures that these coincide.

To be precise, the set of all face maps that are \emph{applicable} to a positive box (as if it were a complete cube),
$\pafm{I}{y}$, with parameters $I$ and $y$ as above, is
\begin{equation}
  \label{def:pafm}
  \pafm{I}{y} \eqdef \union{\thesetst{\spec{i}{b}}{i\in I\ \text{and}\ b\in\fset{\zero,\one}}}{\fset{\spec{y}{\zero}}}
\end{equation}
and for a negative box the set, written $\nafm{I}{y}$, is
\begin{equation}
  \label{def:nafm}
  \nafm{I}{y} \eqdef \union{\thesetst{\spec{i}{b}}{i\in I\ \text{and}\ b\in\fset{\zero,\one}}}{\fset{\spec{y}{\one}}}.
\end{equation}
Intuitively, they represent all the included faces (including the one in the filling dimension).  All the face maps in
these sets are ``polymorphic'' in ambient dimensions in order to simplify various definitions.  For example, the face
map $\parens{\spec{z}{\one}}$ in $\pafm{I}{y}$ for some $z\in I$ can be considered a morphism from $\ext{L}{z}$
to $L$ for any ambient dimensions $L$; in particular, it can be viewed as a morphism from $\bdcomb{I}{J}{y}$
to $\bdcomb{\parens{I\setminus z}}{J}{y}$ for any extra dimensions $J$ by setting $L = \bdcomb{\parens{I\setminus z}}{J}{y}$.
In general, for any dimensions $I$ and $y$, every face map $f\in\pafm{I}{y}$ determines
a set of dimensions $K$ such that $f$ can be seen as a morphism from $\bdcomb{I}{J}{y}$ to $\comb{K}{J}$ for any extra dimensions $J$.
While the extra dimensions $J$ can be inferred from the context, we sometimes write $f_J$ explicitly to indicate
the particular $J$-instance of $f\in\pafm{I}{y}$ with extra dimensions $J$.
All $J$-instances of face maps in $\pafm{I}{y}$ share the same domain $\bdcomb{I}{J}{y}$.

The \emph{positive free-standing box}, $\fspbox{I}{J}{y}$, with parameters $I$, $y$, and $J$ is, as a \emph{co-sieve},
the \emph{saturation} of the $J$-instances of face maps in $\pafm{I}{y}$.
A \emph{co-sieve} is a subfunctor (that is, a subobject in the functor category) of a co-representable functor,
in this case of the free-standing cube with dimensions $\bdcomb{I}{J}{y}$ defined by Equation~\eqref{def:fscube}.
The \emph{saturation} of a set of morphisms $S$ sharing the same domain, intuitively, is the closure of $S$
under post-composition with arbitrary morphisms.  More precisely, being a co-sieve, the saturation of the set $S$
sends dimensions $K$ to the set in which each element is some morphism in $S$ post-composed with some morphism targeted at $K$,
and acts functorially by post-composition.
With these definitions expanded, the positive free-standing box $\fspbox{I}{J}{y}$, as a cubical set,
was defined by the following equations:
\begin{align}
  \label{def:fspbox1}
  \parens{\fspbox{I}{J}{y}}_K & \eqdef
  \thesetst{f:\bdcomb{I}{J}{y}\to K}{f = \ocomp{g_J}{h}\ \text{for some}\ g\in\pafm{I}{y}\ \text{and}\ h}\\
  \label{def:fspbox2}
  \parens{\fspbox{I}{J}{y}}_f\parens{g} & \eqdef \ocomp{g}{f}  \qquad\qquad (f:K\to K', g:\bdcomb{I}{J}{y}\to K).
\end{align}
The \emph{negative} free-standing box, $\fsnbox{I}{J}{y}$, is defined similarly,
albeit the set $\pafm{I}{y}$ replaced by the set $\nafm{I}{y}$.
% Without loss of generality we may assume that every co-sieve of a free-standing cube is a cubical subset of it, simply
% by choosing convenient representatives of the equivalence class of a sub-object in the category of sets,
% for example the representative co-sieve given by Equations~\eqref{def:fspbox1}~and~\eqref{def:fspbox2}.
The \emph{standard positive (resp., negative) free-standing box}, $\fspbox{I}{\emp}{y}$ (resp., $\fsnbox{I}{\emp}{y}$), disallows any
extra dimensions~\citep{williamson2012combinatorial}; the only omitted face is that opposite to the starting face of the
box.

Just as the free-standing cube may be used to specify a geometric cube in $\cs{X}$, a free-standing box may be used to
specify a geometric box in $\cs{X}$.  Specifically, a \emph{positive (resp., negative) geometric box} in a cubical set
$\cs{X}$ is a morphism $\beta:\fspbox{I}{J}{y}\to\cs{X}$ (resp., $\beta:\fsnbox{I}{J}{y}\to\cs{X}$).  From this arises
the \emph{geometric box projection} which projects out a positive (resp., negative) geometric box from a geometric cube $\kappa$
by pre-composition with the inclusion of the free-standing box into the free-standing cube.
\begin{equation*}
  \begin{tikzcd}
    \fspbox{I}{J}{y} \arrow[bend left=40]{rr}{\beta} \arrow[hook]{r}{\iota_0}
    & \fscube{\ext{I}{y}} \arrow{r}{\kappa} & \cs{X}
    &
    \fsnbox{I}{J}{y} \arrow[bend left=40]{rr}{\beta} \arrow[hook]{r}{\iota_1}
    & \fscube{\ext{I}{y}} \arrow{r}{\kappa} & \cs{X}
  \end{tikzcd}
\end{equation*}
The \emph{standard geometric Kan condition} (that is, for boxes with no extra dimensions) for a cubical set $\cs{X}$ states that every
standard geometric box $\beta$ in $\cs{X}$ may be \emph{filled}, or \emph{completed}, to a geometric cube $\kappa$ in
$\cs{X}$ by extending $\beta$ along the inclusion of the free-standing box into the free-standing cube; see
Figure~\ref{fig:gkc}.  This condition amounts to saying that the geometric box projection, restricted to empty $J$, has a
section.  In topological spaces the lifting property holds because the topological cube may be retracted onto any of its
standard contained boxes.  A cubical set satisfying that geometric Kan condition is a \emph{standard geometric Kan
  complex}.
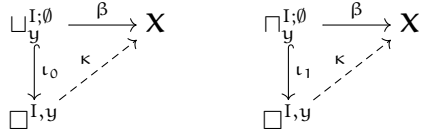
\begin{figure}[ht]
  \centering
  \begin{tikzcd}
    \fspbox{I}{\emp}{y} \arrow{r}{\beta} \arrow[hook]{d}{\iota_0} & \cs{X} \\
    \fscube{\ext{I}{y}} \arrow[dashed]{ur}{\kappa}
  \end{tikzcd}
  \qquad
  \begin{tikzcd}
    \fsnbox{I}{\emp}{y} \arrow{r}{\beta} \arrow[hook]{d}{\iota_1} & \cs{X} \\
    \fscube{\ext{I}{y}} \arrow[dashed]{ur}{\kappa}
  \end{tikzcd}
  \caption{Standard Geometric Kan Condition}
  \label{fig:gkc}
\end{figure}

When extra dimensions $J$ are allowed, the standard geometric Kan condition should be generalized to the \emph{uniform
geometric Kan condition}, which not only fills any geometric box but also relates the filler of
any $\ext{I}{y}$-preserving cubical aspect of a box to the same aspect of the filler.%
\footnote{Note that the original uniformity condition proposed by~\cite{bch} also demands the fillings to respect permutations
of included dimensions, which is implicit in our presentation because we assume $\alpha$-equivalence everywhere.}
Let the \emph{augmentation} of a morphism $h:J\to J'$ with dimensions $I$ distinct from $J$ and $J'$, written $\comb{I}{h}$,
be a morphism from $\comb{I}{J}$ to $\comb{I}{J'}$ which sends $i\in I$ to $i$ and $j\in J$ to $h(j)$.
The free-standing boxes ($\fspbox{I}{J}{y}$ and $\fsnbox{I}{J}{y}$) and free-standing
cubes ($\fscube{\bdcomb{I}{J}{y}}$) all behave functorially in $J$,
in that the functorial action of a morphism $h:J\to J'$ is pre-composition with the augmentation
$\bdcomb{I}{h}{y}:\bdcomb{I}{J}{y}\to\bdcomb{I}{J'}{y}$.
The uniformity means that the filling operation
from $\beta$ to $\kappa$ is natural in $J$ in the sense of the commutative diagrams given in Figure~\ref{fig:uni-gkc} in
which $h:J\to J'$ is an arbitrary morphism from $J$ to $J'$ in $\CC$.
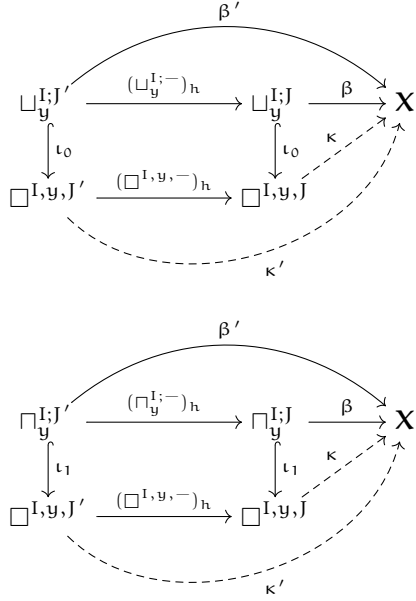
\begin{figure}[tp]
  \centering
  \begin{tikzcd}
    \fspbox{I}{J'}{y}
    \arrow{rr}{\parens{\fspbox{I}{-}{y}}_h}
    \arrow[hook]{d}{\iota_0}
    \arrow[bend left=40]{rrr}{\beta'}
    &&
    \fspbox{I}{J}{y} \arrow{r}{\beta}
    \arrow[hook]{d}{\iota_0}
    &
    \cs{X}
    \\
    \fscube{\bdcomb{I}{J'}{y}}
    \arrow{rr}{\parens{\fscube{\bdcomb{I}{-}{y}}}_h}
    \arrow[bend right=60, dashed]{rrru}[swap]{\kappa'}
    &&
    \fscube{\bdcomb{I}{J}{y}}
    \arrow[dashrightarrow]{ur}{\kappa}
  \end{tikzcd}
  \qquad
  \begin{tikzcd}
    \fsnbox{I}{J'}{y}
    \arrow{rr}{\parens{\fsnbox{I}{-}{y}}_h}
    \arrow[hook]{d}{\iota_1}
    \arrow[bend left=40]{rrr}{\beta'}
    &&
    \fsnbox{I}{J}{y} \arrow{r}{\beta}
    \arrow[hook]{d}{\iota_1}
    &
    \cs{X}
    \\
    \fscube{\bdcomb{I}{J'}{y}}
    \arrow{rr}{\parens{\fscube{\bdcomb{I}{-}{y}}}_h}
    \arrow[bend right=60, dashed]{rrru}[swap]{\kappa'}
    &&
    \fscube{\bdcomb{I}{J}{y}}
    \arrow[dashed]{ur}{\kappa}
  \end{tikzcd}
  \caption{Uniform Geometric Kan Condition}
  \label{fig:uni-gkc}
\end{figure}

To obtain a more combinatorial formulation of the Kan
structure, it is useful to formulate an algebraic representation of geometric boxes, much as the Yoneda Lemma provides
an algebraic representation of geometric cubes.  Writing $\algpbox{\cs{X}}{I}{J}{y}$ for a suitably algebraic
representation of \emph{positive algebraic boxes} in $\cs{X}$, the intention is that there be a bijection
\begin{equation}
  \label{eq:pb}
  \hhpb{I}{J}{y} : \hom{\CSET}{\fspbox{I}{J}{y}}{\cs{X}} \cong \algpbox{\cs{X}}{I}{J}{y},
\end{equation}
and, analogously, for there to be a suitably algebraic representation $\algnbox{\cs{X}}{I}{J}{y}$ of the set of
\emph{negative algebraic boxes} for which there is a bijection
\begin{equation}
  \label{eq:nb}
  \hhnb{I}{J}{y} : \hom{\CSET}{\fsnbox{I}{J}{y}}{\cs{X}} \cong \algnbox{\cs{X}}{I}{J}{y}.
\end{equation}
Moreover, these bijections should be natural in $J$.  These bijections should be compared to the one given by the Yoneda
Lemma for the free-standing cubes,
\begin{equation}
  \label{eq:c}
  \coyoneda{I} : \hom{\CSET}{\fscube{I}}{\cs{X}} \cong \cs{X}_I,
\end{equation}
which is natural in $I$.

To derive a suitable definition for $\algpbox{\cs{X}}{I}{J}{y}$ and $\algnbox{\cs{X}}{I}{J}{y}$, it is helpful to review
the definition of the free-standing boxes on given shape parameters $I$ and $y$ given by
Equations~\eqref{def:fspbox1} and~\eqref{def:fspbox2}.  Because the free-standing box $\fspbox{I}{J}{y}$ is the saturation
of (the $J$-instances in) $\pafm{I}{y}$, an algebraic representation of a positive geometric box $\beta :
\fspbox{I}{J}{y} \to \cs{X}$ may be expected to be determined by a family of (lower-dimensional) cubes $\alg{\beta}_f =
\beta_{\cod{f_J}}\parens{f_J}$ for each $f \in \pafm{I}{y}$.  We say two morphisms $f_1$ and $f_2$ are
\emph{reconcilable} if $\ocomp{f_1}{g_1} = \ocomp{f_2}{g_2}$ for some $g_1$ and $g_2$.
Two polymorphic face maps $f_1$ and $f_2$ are reconcilable if their compatible instances are reconcilable.
The naturality of $\beta$ guarantees the following:
\begin{equation}
  \label{eq:algbox:naive}
  \text{For any}\ f_1, f_2 \in \pafm{I}{y}\ \text{reconcilable by $g_1$ and $g_2$,}\ \cs{X}_{g_1}\parens{\alg{\beta}_{f_1}} = \cs{X}_{g_2}\parens{\alg{\beta}_{f_2}}.
\end{equation}
That is, if two cubes, $f_1$ and $f_2$, in the free-standing cube $\fscube{\bdcomb{I}{J}{y}}$
have coincident aspects $g_1$ and $g_2$, then so does the box $\beta$.
It turns out that Equation~\eqref{eq:algbox:naive} is the critical condition
to make the family of cubes $\cs{\beta}$ qualified as a box.

Because $\pafm{I}{y}$ consists solely of face maps, Equation~\eqref{eq:algbox:naive} can be further restricted;
it is not necessary to check all possible $g_1$'s and $g_2$'s.
We say two face maps $f_1$ and $f_2$ are \emph{orthogonal} if $\ocomp{f_2}{f_1} = \ocomp{f_1}{f_2}$ holds,
or equivalently they are reconcilable by some instances of $f_2$ and $f_1$.%
\footnote{Note that two occurrences $f_1$ in the equation $\ocomp{f_2}{f_1} = \ocomp{f_1}{f_2}$
  refer to different instances of the polymorphic face maps $f_1$.  So does $f_2$.}
  As we shall see, Equation~\eqref{eq:algbox:naive} is equivalent to the following \emph{adjacency condtion for positive boxes}:
\begin{equation}
  \label{def:adjacency:pos}
  \text{For any orthogonal}\ f_1, f_2 \in \pafm{I}{y},\ \cs{X}_{f_2}\parens{\alg{\beta}_{f_1}} = \cs{X}_{f_1}\parens{\alg{\beta}_{f_2}}.
\end{equation}
The intuition is that, given several cubes that should be faces of some cube,
any attaching part must be shared by at least two faces-to-be, and so it is sufficient to check whether any two
faces-to-be fit together.

Equation~\eqref{def:adjacency:pos} is a special case of Equation~\eqref{eq:algbox:naive}
where $g_1$ and $g_2$ are restricted to instances of $f_2$ and $f_1$, respectively.
It suffices to show that Equation~\eqref{def:adjacency:pos} implies Equation~\eqref{eq:algbox:naive}.
Assuming Equation~\eqref{def:adjacency:pos}
and there are two polymorphic face maps $f_1$ and $f_2$ in $\pafm{I}{y}$
that are reconcilable by some $g_1$ and $g_2$,
the goal is to show that $\cs{X}_{g_1}\parens{\alg{\beta}_{f_1}} = \cs{X}_{g_2}\parens{\alg{\beta}_{f_2}}$.
Recall that any morphism admits a canonical form that consist of face maps, renamings, and degeneracies, in that order.
Here consider a canonical form of the morphism $\ocomp{f_1}{g_1} = \ocomp{f_2}{g_2}$.
If $f_1$ and $f_2$ are the same face map,
then the canonical form implies that $g_1$ and $g_2$ are also the same,
which implies $\cs{X}_{g_1}\parens{\alg{\beta}_{f_1}} = \cs{X}_{g_2}\parens{\alg{\beta}_{f_2}}$ immediately.
Without loss of generality, suppose $f_1$ and $f_2$ are different face maps.
Because both $f_1$ and $f_2$ appear in the canonical form of $\ocomp{f_1}{g_1} = \ocomp{f_2}{g_2}$,
they must be orthogonal;
by Equation~\eqref{def:adjacency:pos} $\cs{X}_{f_2}\parens{\alg{\beta}_{f_1}} = \cs{X}_{f_1}\parens{\alg{\beta}_{f_2}}$.
Moreover, the canonical forms of $g_1$ and $g_2$ must differ by exactly one face map,
where $f_1$ is missing in $g_1$ and $f_2$ is missing in $g_2$,
which is to say that $g_1$ and $g_2$ factors through a ``common morphism'' $g_{1,2}$ such that
\begin{gather}
  g_1 = \ocomp{f_2}{g_{1,2}}
  \\
  g_2 = \ocomp{f_1}{g_{1,2}}.
\end{gather}
By functoriality of $\cs{X}$,
\begin{equation}
  \cs{X}_{g_1}\parens{\alg{\beta}_{f_1}}
  =
  \cs{X}_{g_{1,2}}\parens{\cs{X}_{f_2}\parens{\alg{\beta}_{f_1}}}
  =
  \cs{X}_{g_{1,2}}\parens{\cs{X}_{f_1}\parens{\alg{\beta}_{f_2}}}
  =
  \cs{X}_{g_2}\parens{\alg{\beta}_{f_2}}.
\end{equation}

\smallskip

More precisely, a positive algebraic box with parameters $I$, $y$, and $J$ is defined to be a family
of cubes $\alg{\beta}_f \in \cs{X}_{\cod{f_J}}$ indexed by the face maps $f \in \pafm{I}{y}$
from dimensions $\bdcomb{I}{J}{y}$ to lower dimensions
such that the adjacency condition~\eqref{def:adjacency:pos} holds.
A negative algebraic box is defined similarly, except that every $\pafm{I}{y}$ is replaced by $\nafm{I}{y}$.
The set $\algpbox{X}{I}{J}{y}$ (resp., $\algnbox{X}{I}{J}{y}$) of \emph{positive} (resp., \emph{negative})
\emph{algebraic $\ext{I}{y}$-shaped boxes with extra dimensions $J$} is defined to be the collection of all such families of cubes
$\famset{\alg{\beta}_f}{f\in\pafm{I}{y}}$ (resp., $\famset{\alg{\beta}_f}{f\in\nafm{I}{y}}$).

The definition of a positive algebraic box with extra dimensions $J$ may be extended to a functor in $J$, giving rise to
the cubical set $\algpbox{X}{I}{-}{y}$ of all $\ext{I}{y}$-shaped boxes in $\cs{X}$:
\begin{align}
  \parens{\algpbox{X}{I}{-}{y}}_J &\eqdef{} \algpbox{X}{I}{J}{y}  && (J:\CC) \\
  \parens{\algpbox{X}{I}{-}{y}}_h(\cs{\beta})_f &\eqdef{} \parens{\cs{X}_{\parens{\comb{K}{-}}}}_h \parens{\alg{\beta}_f} && (h:J\to J', \parens{f : \bdcomb{I}{J''}{y} \to \comb{K}{J''}} \in \pafm{I}{y})
\end{align}
An analogous definition may be given of the cubical set $\algnbox{\cs{X}}{I}{-}{y}$ by replacing $\pafm{I}{y}$ by
$\nafm{I}{y}$ in the specification of its functorial action.  Observe that in both cases the functorial action ``maps''
the action of $h$ over every cube in the collection of cubes that compose an algebraic box.

The \emph{positive algebraic box projection} is a morphism of cubical sets
\begin{displaymath}
  \paproj{I}{J}{y}:\cs{X}_{\bdcomb{I}{J}{y}}\to\algpbox{\cs{X}}{I}{J}{y}
\end{displaymath}
that sends $\alg{\kappa}\in\cs{X}_{\bdcomb{I}{J}{y}}$ to $\famset{\cs{X}_f\parens{\kappa}}{f\in\pafm{I}{y}}$ by
selecting from the cube the boundary faces determined by the shape of the box.  Similarly, the \emph{negative algebraic
  box projection}
\begin{displaymath}
  \naproj{I}{J}{y}:\cs{X}_{\bdcomb{I}{J}{y}}\to\algnbox{\cs{X}}{I}{J}{y}
\end{displaymath}
sends $\alg{\kappa}\in\cs{X}_{\bdcomb{I}{J}{y}}$ to $\famset{\cs{X}_f\parens{\kappa}}{f\in\nafm{I}{y}}$.

A \emph{positive (resp., negative) uniform algebraic filling operation} is a \emph{section} of a \emph{positive (resp.,
  negative) algebraic box projection}:
\begin{displaymath}
  \pfill{X}{I}{J}{y} : \algpbox{X}{I}{J}{y}\to\cs{X}_{\bdcomb{I}{J}{y}},
\end{displaymath}
resp.,
\begin{displaymath}
  \nfill{X}{I}{J}{y} : \algnbox{\cs{X}}{I}{J}{y}\to\cs{X}_{\bdcomb{I}{J}{y}}.
\end{displaymath}
The filling operations choose an algebraic cube that fills each algebraic box in $\cs{X}$, naturally in the extra
dimensions $J$.  The \emph{uniform algebraic Kan condition} for a cubical set $\cs{X}$ states that such uniform filling
operations exist for $\cs{X}$.

\paragraph{Equivalence between algebraic and geometric uniform Kan structures.}
Now that we have phrased the uniform Kan structure, which consists of the cubes, the boxes,
the box projections, and the uniform filling operations, in an algebraic manner,
it is important to show that the algebraic and geometric formulations are equivalent,
in the sense that the cubes, the boxes, and the box projections all match up
and there is an equivalence between uniform filling operations in two formulations.
The algebraic and geometric cubes ($\cs{X}$ and $\hom{\CSET}{\fscube{-}}{X}$) are already
aligned by the Yoneda Lemma.  What is lacking are the natural bijections
between algebraic and geometric boxes (as expressed by \eqref{eq:pb}~and~\eqref{eq:nb})
such that the box projections and the uniform filling operations are related.
The equivalence is summarized in Figure~\ref{fig:gkc-akc}.

\begin{figure}[ht]
  \centering
  \begin{tikzcd}[
      bijection/.style={Leftrightarrow}
    ]
    &
    \algpbox{X}{I}{J}{y}
    \arrow{rr}{\parens{\algpbox{X}{I}{-}{y}}_h}
    \arrow[dashed, bend left]{dd}[near start]{\pfill{}{I}{J}{y}}
    \arrow[leftarrow]{dd}[swap, near start]{\paproj{I}{J}{y}}
    & &
    \algpbox{X}{I}{J'}{y}
    \arrow[dashed, bend left]{dd}[near start]{\pfill{}{I}{J'}{y}}
    \arrow[leftarrow]{dd}[near end, swap]{\paproj{I}{J'}{y}}
    \\
    \hom{\CSET}{\fspbox{I}{J}{y}}{\cs{X}}
    \arrow[bijection]{ur}{\hhp}
    \arrow[crossing over]{rr}[very near start, swap, yshift=-0.3em]{\hom{\CSET}{\fspbox{I}{-}{y}}{\cs{X}}_h}
    \arrow[dashed, bend left]{dd}[pos=0.45]{\plift{}{I}{J}{y}}
    \arrow[leftarrow]{dd}[near end, swap]{\ocomp{\iota_0}{-}}
    & &
    \hom{\CSET}{\fspbox{I}{J'}{y}}{\cs{X}}
    \arrow[bijection]{ur}{\hhp}
    \\
    &
    \cs{X}_{\bdcomb{I}{J}{y}}
    \arrow{rr}[near start]{\parens{\cs{X}_{\parens{\bdcomb{I}{-}{y}}}}_h}
    & &
    \cs{X}_{\bdcomb{I}{J'}{y}}
    \\
    \hom{\CSET}{\fscube{\bdcomb{I}{J}{y}}}{\cs{X}}
    \arrow[bijection]{ur}{\coyoneda{}}
    \arrow{rr}{\hom{\CSET}{\fscube{\bdcomb{I}{-}{y}}}{\cs{X}}_h}
    & &
    \hom{\CSET}{\fscube{\bdcomb{I}{J'}{y}}}{\cs{X}}
    \arrow[bijection]{ur}{\coyoneda{}}
    \arrow[leftarrow, dashed, crossing over, bend right]{uu}[swap, near end]{\plift{X}{I}{J}{y}}
    \arrow[crossing over]{uu}[near start]{\ocomp{\iota_0}{-}}
  \end{tikzcd}

  \bigskip

  \begin{tabular}{lllll}
    \textit{Elements} & \textit{Geometric} & \textit{Algebraic} & \textit{Theorem} & \textit{Reference}
    \\
    \midrule
    Cubes & $\fscube{I} \to \cs{X}$       & $\cs{X}_I$             & Bijection & Yoneda
    \\
    & & & Naturality (see boxes) & Ext. to Yoneda
    \\
    Positive Boxes & $\fspbox{I}{J}{y} \to \cs{X}$ & $\algpbox{X}{I}{J}{y}$ & Bijection & Prop.~\ref{prop:box:biject}
    \\
    & & & Naturality in $J$ & Prop.~\ref{prop:box:biject:nat}.
    \\
    Negative Boxes & $\fsnbox{I}{J}{y} \to \cs{X}$ & $\algnbox{X}{I}{J}{y}$ & Bijection & Prop.~\ref{prop:box:biject}
    \\
    & & & Naturality in $J$ & Prop.~\ref{prop:box:biject:nat}.
    \\
    Pos.\ Box Projection & $\ocomp{\iota_0}{-}$ & $\paproj{I}{J}{y}$ & Identity & Prop.~\ref{prop:boxproj}
    \\
    Neg.\ Box Projection & $\ocomp{\iota_1}{-}$ & $\naproj{I}{J}{y}$ & Identity & Prop.~\ref{prop:boxproj}
    \\
    Pos.\ Uniform Filling & $\plift{X}{I}{J}{y}$ & $\pfill{X}{I}{J}{y}$ & Type Equivalence & Prop.~\ref{thm:gkc-akc}
    \\
    Neg.\ Uniform Filling & $\nlift{X}{I}{J}{y}$ & $\nfill{X}{I}{J}{y}$ & Type Equivalence & Prop.~\ref{thm:gkc-akc}
  \end{tabular}
  \caption{Equivalence between Algebraic and Geometric Kan Structures}
  \label{fig:gkc-akc}
\end{figure}

The natural bijections between geometric and algebraic boxes are pairs of functions, one computing (algebraic) nerve of
a geometric box, written $\pnerve{X}{I}{J}{y}$ and $\nnerve{X}{I}{J}{y}$ for their positive and negative forms, and the
other computing the (geometric) realization of an algebraic box, written $\prealize{X}{I}{J}{y}$ and
$\nrealize{X}{I}{J}{y}$ for their positive and negative forms.  The types of these functions are as follows:
\begin{align*}
  \pnerve{X}{I}{J}{y} &: \hom{\CSET}{\fspbox{I}{J}{y}}{\cs{X}} \to \algpbox{X}{I}{J}{y}
  \\
  \nnerve{X}{I}{J}{y} &: \hom{\CSET}{\fsnbox{I}{J}{y}}{\cs{X}} \to \algnbox{X}{I}{J}{y}
  \\[2ex]
  \prealize{X}{I}{J}{y} &: \algpbox{X}{I}{J}{y} \to \hom{\CSET}{\fspbox{I}{J}{y}}{\cs{X}}
  \\
  \nrealize{X}{I}{J}{y} &: \algnbox{X}{I}{J}{y} \to \hom{\CSET}{\fsnbox{I}{J}{y}}{\cs{X}}
\end{align*}
They are defined by the following equations:
\begin{align}
  \pnerve{X}{I}{J}{y}\parens{\beta}_f &\eqdef{} \beta_{\cod{f}}\parens{f}
  \qquad\qquad (f\in\pafm{I}{y})
  \\
  \nnerve{X}{I}{J}{y}\parens{\beta}_f &\eqdef{} \beta_{\cod{f}}\parens{f}
  \qquad\qquad (f\in\nafm{I}{y})
  \\[2ex]
  \parens{\prealize{X}{I}{J}{y}\parens{\alg{\beta}}}_K\parens{f} &\eqdef{} \cs{X}_{f_2}\parens{\alg{\beta}_{f_1}}
  \qquad\qquad (f:\parens{\bdcomb{I}{J}{y}\to K}, f=\ocomp{f_1}{f_2}, f_1\in\pafm{I}{y})
  \\
  \parens{\nrealize{X}{I}{J}{y}\parens{\alg{\beta}}}_K\parens{f} &\eqdef{} \cs{X}_{f_2}\parens{\alg{\beta}_{f_1}}
  \qquad\qquad (f:\parens{\bdcomb{I}{J}{y}\to K}, f=\ocomp{f_1}{f_2}, f_1\in\nafm{I}{y})
\end{align}

\begin{proposition}
  \label{prop:nerve-welldef}
  The function $\pnerve{X}{I}{J}{y}$ define valid positive algebraic boxes in the sense that the adjacency
  condition~\eqref{def:adjacency:pos} holds.
  So does $\nnerve{X}{I}{J}{y}$ except $\pafm{I}{y}$ is replaced by $\nafm{I}{y}$.
\end{proposition}

\begin{proof}
  We only deal with positive boxes.
  Let $\beta : \fspbox{I}{J}{y} \to X$ be a geometric box
  and $\alg{\beta}$ be $\pnerve{X}{I}{J}{y}\parens{\beta}$.
  For any $f_1, f_2 \in \pafm{I}{y}$ such that $\ocomp{f_1}{f_2} = \ocomp{f_2}{f_1} = g$,
  \begin{align*}
    \cs{X}_{f_2}\parens{\alg{\beta}_{f_1}}
    &=
    \cs{X}_{f_2}\parens{\beta_{\cod{f_1}}\parens{f_1}}
    &&\text{(definition of $\alg{\beta} = \pnerve{X}{I}{J}{y}\parens{\beta}$)}
    \\&=
    \beta_{\cod{g}}\parens{\parens{\fspbox{I}{J}{y}}_{f_2}\parens{f_1}}
    &&\text{(naturality of $\beta$)}
    \\&=
    \beta_{\cod{g}}\parens{\ocomp{f_1}{f_2}}
    &&\text{(functorial action of co-sieves)}
    \\&=
    \beta_{\cod{g}}\parens{\ocomp{f_2}{f_1}}
    &&\text{(assumption)}
    \\&=
    \beta_{\cod{g}}\parens{\parens{\fspbox{I}{J}{y}}_{f_1}\parens{f_2}}
    &&\text{(functorial action of co-sieves)}
    \\&=
    \cs{X}_{f_1}\parens{\beta_{\cod{f_2}}\parens{f_2}}
    &&\text{(naturality of $\beta$)}
    \\&=
    \cs{X}_{f_1}\parens{\alg{\beta}_{f_2}}
    &&\text{(definition of $\alg{\beta} = \pnerve{X}{I}{J}{y}$)}
  \end{align*}
  and thus $\alg{\beta}$ is a positive algebraic box.
\end{proof}

\begin{proposition}
  \label{prop:realize-welldef}
  The functions $\prealize{X}{I}{J}{y}$ and $\nrealize{X}{I}{J}{y}$ are well-defined.
\end{proposition}

\begin{proof}
  We only show the case of positive boxes.  Suppose $\alg{\beta} \in \algpbox{X}{I}{J}{y}$ is an algebraic box.  Let
  $\beta = \prealize{X}{I}{J}{y}\parens{\alg{\beta}}$.  Note that any morphism $f \in \parens{\fspbox{I}{J}{y}}_{K}$, by
  definition, admits a factorization $f = \ocomp{f_1}{f_2}$ for some $f_1 \in \pafm{I}{y}$, and thus the side condition
  on the definition of $\prealize{X}{I}{J}{y}$ can always be satisfied.  The worry is that the same morphism $f$ may
  admit multiple different such factorizations, and it is necessary to show the resulting cube is nevertheless properly
  defined.  So, suppose that $f = \ocomp{f_1}{f_2} = \ocomp{f_1'}{f_2'}$, for some $f_1$ and $f_1'\in\pafm{I}{y}$, in
  order to show that
  \begin{displaymath}
    \label{eq:well-defined}
    \cs{X}_{f_2}\parens{\beta_{\cod{f_1}}\parens{f_1}} = \cs{X}_{f_2'}\parens{\beta_{\cod{f_1'}}\parens{f_1'}}.
  \end{displaymath}
  The reasoning is similar to the intuition we gave earlier.  The morphism $f$ admits a canonical form that consist of
  face maps, renamings, and degeneracies, in that order.  If $f_1$ and $f_1'$ are the same face map, then the canonical
  form implies that $f_2$ and $f_2'$ must be the same, because face maps are epimorphisms, from which
  Equation~\eqref{eq:well-defined} follows immediately.  Otherwise, assume $f_1$ and $f_1'$ are different face maps.
  The canonical forms of $f_2$ and $f_2'$ must differ by exactly one face map, where $f_1$ is missing in $f_1$ and
  $f_1'$ is missing in $f_2'$, which is to say that $f_2$ and $f_2'$ share a ``common morphism'' $f_3$ such that
  \begin{displaymath}
    f_2 = \ocomp{f_1'}{f_3}\ \text{and}\ f_2' = \ocomp{f_1}{f_3}.
  \end{displaymath}
  Therefore
  \begin{align*}
    \cs{X}_{f_2}\parens{\alg{\beta}_{f_1}}
    &=
    \cs{X}_{\ocomp{f_1'}{f_3}}\parens{\alg{\beta}_{f_1}}
    &&\text{(construction of $f_3$)}
    \\&=
    \cs{X}_{f_3}\parens{\cs{X}_{f_1'}\parens{\alg{\beta}_{f_1}}}
    &&\text{(functoriality of $\cs{X}$)}
    \\&=
    \cs{X}_{f_3}\parens{\cs{X}_{f_1}\parens{\alg{\beta}_{f_1'}}}
    &&\text{(adjacency condition of algebraic boxes)}
    \\&=
    \cs{X}_{\ocomp{f_1}{f_3}}\parens{\alg{\beta}_{f_1}}
    &&\text{(functoriality of $\cs{X}$)}
    \\&=
    \cs{X}_{f_2'}\parens{\alg{\beta}_{f_1'}}
    &&\text{(construction of $f_3$)}
  \end{align*}
  which proves that $\beta_{K}\parens{f}$ is well-defined.  Naturality of $\beta$ is proved by considering an
  arbitrary morphism $g : K \to K'$ as follows:
  \begin{align*}
    \cs{X}_{g}\parens{\beta_{K}\parens{f}}
    &=
    \cs{X}_{g}\parens{\beta_{K}\parens{\ocomp{f_1}{f_2}}}
    &&\text{(decomposition of $f$)}
    \\&=
    \cs{X}_{g}\parens{\cs{X}_{f_2}\parens{\alg{\beta}_{f_1}}}
    &&\text{(definition of $\prealize{X}{I}{J}{y}$)}
    \\&=
    \cs{X}_{\ocomp{f_2}{g}}\parens{\alg{\beta}_{f_1}}
    &&\text{(functoriality of $\cs{X}$)}
    \\&=
    \beta_{K}\parens{\ocomp{f_1}{\ocomp{f_2}{g}}}
    &&\text{(definition of $\prealize{X}{I}{J}{y}$)}
    \\&=
    \beta_{K}\parens{\ocomp{f}{g}}
    &&\text{(decomposition of $f$)}
    \\&=
    \beta_{K}\parens{\parens{\fspbox{I}{J}{y}}_{g}\parens{f}}.
    &&\text{(functorial action of $\fspbox{I}{J}{y}$)}
  \end{align*}
\end{proof}

\begin{proposition}
  \label{prop:box:biject}
  $\pnerve{X}{I}{J}{y}$ and $\prealize{X}{I}{J}{y}$ are inverse to each other;
  so are $\nnerve{X}{I}{J}{y}$ and $\nrealize{X}{I}{J}{y}$.
\end{proposition}

\begin{proof}
  As usual, it suffices to consider only positive boxes; the argument for negative boxes is analogous.  Let $\alg{\beta}
  \in \algpbox{X}{I}{J}{y}$ be a positive algebraic box, let $\alg{\beta}' =
  \pnerve{X}{I}{J}{y}\parens{\prealize{X}{I}{J}{y}\parens{\alg{\beta}}}$ be the nerve of its realization, and let the
  intermediate geometic box $\beta = \prealize{X}{I}{J}{y}\parens{\alg{\beta}}$.  To show that $\alg{\beta}
  = \alg{\beta}'$, consider any $f \in \pafm{I}{y}$, and calculate as follows:
  \begin{align*}
    \alg{\beta}'_f
    &=
    \beta_K\parens{f}
    &&\text{(definition of $\pnerve{X}{I}{J}{y}$)}
    \\&=
    \cs{X}_{\id}\parens{\alg{\beta}_f}
    &&\text{(definition of $\prealize{X}{I}{J}{y}$)}
    \\&=
    \alg{\beta}_f.
    &&\text{(functoriality of $\cs{X}$)}
  \end{align*}

  Conversely, let $\beta : \fspbox{I}{J}{y} \to \cs{X}$ be a positive geometric box, let $\beta' =
  \prealize{X}{I}{J}{y}\parens{\pnerve{X}{I}{J}{y}\parens{\beta}}$ be the realization of its nerve, and let $\alg{\beta}
  = \pnerve{X}{I}{J}{y}\parens{\beta}$ be the intermediate algebraic box in the equation.  For any $f =
  \ocomp{f_1}{f_2}$ such that $f_1 \in \pafm{I}{y}$,
  \begin{align*}
    \beta'_{K}\parens{f}
    &=
    \cs{X}_{f_2}\parens{\alg{\beta}_{f_1}}
    &&\text{(definition of $\prealize{X}{I}{J}{y}$)}
    \\&=
    \cs{X}_{f_2}\parens{\beta_{K'}\parens{f_1}}
    &&\text{(definition of $\pnerve{X}{I}{J}{y}$)}
    \\&=
    \beta_{K}\parens{\parens{\fspbox{I}{J}{y}}_{f_2}\parens{f_1}}
    &&\text{(naturality of $\beta$)}
    \\&=
    \beta_{K}\parens{\ocomp{f_1}{f_2}}
    &&\text{(functorial action of $\fspbox{I}{J}{y}$)}
    \\&=
    \beta_{K}\parens{f}.
    &&\text{(decomposition of $f$)}
  \end{align*}

  Therefore these two functions form a bijection.
\end{proof}

\begin{proposition}
  \label{prop:box:biject:nat}
  The bijections in Proposition~\ref{prop:box:biject} are natural in $J$.
\end{proposition}

\begin{proof}
  Consider, as usual, the case of positive boxes; the negative boxes are handled similarly.  Let $\beta :
  \fspbox{X}{I}{J}{y} \to \cs{X}$ be a positive geometric box.  It suffices to show that, for any morphism $h : J \to
  J'$ in $\CC$,
  \begin{displaymath}
    \parens{\algpbox{X}{I}{-}{y}}_h
    \parens{\pnerve{X}{I}{J}{y}\parens{\beta}}
    =
    \pnerve{X}{I}{J'}{y}
    \parens{\ocomp{\parens{\fspbox{I}{-}{y}}_h}{\beta}}.
  \end{displaymath}
  Focusing on the left hand side,
  for any face map $\parens{f : \bdcomb{I}{J''}{y} \to \comb{K}{J''}} \in \pafm{I}{y}$
  that is polymorphic in $J''$, we know
  \begin{align*}
    \parens{\algpbox{X}{I}{-}{y}}_h
    \parens{\pnerve{X}{I}{J}{y}\parens{\beta}}_f
    &=
    \parens{\cs{X}_{\parens{\comb{K}{-}}}}_h
    \parens{
      \pnerve{X}{I}{J}{y}\parens{\beta}_f
    }
    &&\text{(functorial action of $\algpbox{X}{I}{-}{y}$)}
    \\&=
    \parens{\cs{X}_{\parens{\comb{K}{-}}}}_h
    \parens{\beta_{\comb{K}{J}}\parens{f}}.
    &&\text{(definition of $\pnerve{X}{I}{-}{y}$)}
  \end{align*}
  As for the right hand side,
  \begin{align*}
    \pnerve{X}{I}{J'}{y}
    \parens{\ocomp{\parens{\fspbox{I}{-}{y}}_h}{\beta}}_f
    &=
    \parens{\ocomp{\parens{\fspbox{I}{-}{y}}_h}{\beta}}_{\comb{K}{J'}}
    \parens{f}
    &&\text{(definition of $\pnerve{X}{I}{-}{y}$)}
    \\&=
    \beta_{\comb{K}{J'}}
    \parens{\parens{\parens{\fspbox{I}{-}{y}}_h}_{\comb{K}{J'}}
    \parens{f}}
    &&\text{(composition of natural transformations)}
    \\&=
    \beta_{\comb{K}{J'}}
    \parens{\ocomp{\parens{\bdcomb{I}{-}{y}}_{h}}{f}}
    &&\text{(functorial action of $\fspbox{I}{-}{y}$)}
    \\&=
    \beta_{\comb{K}{J'}}
    \parens{\ocomp{f}{\parens{\comb{K}{-}}_{h}}}
    &&\text{($f$ is polymorphic in dimensions other than $\ext{I}{y}$)}
    \\&=
    \parens{\cs{X}_{\parens{\comb{K}{-}}}}_h
    \parens{\beta_{\comb{K}{J}}\parens{f}}.
    &&\text{(naturality of $\beta$)}
  \end{align*}
\end{proof}

\begin{proposition}
  \label{prop:boxproj}
  The algebraic box projection of an algebraic cube
  corresponds to the geometric box projection of the corresponding geometric cube.
\end{proposition}

\begin{proof}
  As usual, consider the positive boxes, as the negatives are handled analogously.  Suppose $\alg{\kappa}$ is an
  algebraic cube in $\cs{X}_{\bdcomb{I}{J}{y}}$ and $\alg{\beta}$ be its algebraic box projection to
  $\algpbox{X}{I}{J}{y}$.  Let the corresponding geometric cube be $\kappa$ and its geometric box projection be $\beta$.
  It suffices to show that $\pnerve{X}{I}{J}{y}\parens{\beta} = \alg{\beta}$.  For any morphism $\parens{f :
    \bdcomb{I}{J}{y} \to K} \in \pafm{I}{y}$,
  \begin{align*}
    \pnerve{X}{I}{J}{y}\parens{\beta}_f
    &=
    \beta_K\parens{f}
    &&\text{(definition of $\pnerve{X}{I}{J}{y}$)}
    \\&=
    \parens{\ocomp{\iota_0}{\kappa}}_{K}\parens{f}
    &&\text{(definition of geometric box projection)}
    \\&=
    \kappa_K\parens{\parens{\iota_0}_{K}\parens{f}}
    &&\text{(composition of natural transformations)}
    \\&=
    \kappa_K\parens{f}
    &&\text{(inclusion)}
    \\&=
    \cs{X}_f\parens{\alg{\kappa}}
    &&\text{(Yoneda)}
    \\&=
    \alg{\beta}_f.
    &&\text{(definition of algebraic box projection)}
  \end{align*}
\end{proof}

\begin{proposition}
  \label{thm:gkc-akc}
  There is a bijection between the set of geometric uniform filling operations
  and that of algebraic uniform filling operations.
\end{proposition}

\begin{proof}
  We only show the case of positive boxes.
  Suppose we have a positive geometric filling operation $\plift{X}{I}{J}{y}$.
  $\ocomp{\ocomp{\prealize{X}{I}{J}{y}}{\plift{X}{I}{J}{y}}}{\pnerve{X}{I}{J}{y}}$
  is a valid positive algebraic uniform filling operation
  by the Yoneda Lemma and Propositions~\ref{prop:box:biject},~\ref{prop:box:biject:nat}~and~\ref{prop:boxproj}.
  Similarly, suppose we have a positive algebraic filling operation $\pfill{X}{I}{J}{y}$.
  $\ocomp{\ocomp{\pnerve{X}{I}{J}{y}}{\pfill{X}{I}{J}{y}}}{\prealize{X}{I}{J}{y}}$
  is a valid geometric one.

  By the Yoneda Lemma and Proposition~\ref{prop:box:biject}, the above two constructions
  are inverse to each other, which shows that there is a bijection between
  the sets of algebraic and geometric uniform filling operations.
\end{proof}

\citeauthor{williamson2012combinatorial} shows in detail that a Kan complex forms an $\omega$-groupoid in that reversion
and concatenation of cubes may be defined in such a way that the groupoid laws hold up to higher identifications given
by cubes.  Rather than repeat the construction here, the reader is referred to~\cite{williamson2012combinatorial} for a
proof in the standard cubical setting that may be adapted to the present setting using named, rather than numbered,
dimensions.

\section{Kan Fibrations}
\label{sec:kf}

The uniform Kan condition ensures that a cubical set has sufficient structure to be the interpretation of a closed type
in which higher-dimensional cubes are interpreted as identifications.  \citeauthor{williamson2012combinatorial} shows,
by explicit constructions, that a Kan complex forms an $\omega$-groupoid.  More generally, it is necessary to consider
the conditions under which a cubical family may serve as the interpretation of a family of types indexed by a type.
Thinking informally of a family of types as a mapping sending each element of the index type to a type, one quickly
arrives at the requirement that this assignment should respect the identifications expressed by the cubical structure of
the indexing type---identified indices should determine ``equivalent'' types.  This means that if $\alpha_0$ and
$\alpha_1$ are identified by $\alpha$ in the index type, then there should be \emph{transport functions} between the
type assigned to $\alpha_0$ and the type assigned to $\alpha_1$ that are mutually inverse in the sense that the starting
point should be identified with the ending point for both composites.  The purpose of this section is to make these
intuitions precise.

As motivation, let us consider the concept of a family of sets $\famset{Y_x}{x\in X}$, where $X$ is a set of indices
and, for each $x\in X$, $Y_x$ is a set.  What, precisely, is such a thing?  One interpretation is that the family is a
mapping $Y:X\to\SET$ such that for every $x\in X$, $Y\parens{x}=Y_x$.  For this to make sense, $X$ must be construed as
a category, the obvious choice being the discrete category on objects $X$; the functoriality of $Y$ is then trivial.
One may worry that the codomain of $Y$ is a ``large'' category, but such worries may be allayed by recalling that the
axioms of replacement and union ensure that the direct image of $X$ by $Y$ must form a set---we are only using a ``small
part'' of the entire category $\SET$.  But this observation leads immediately to an alternative, and in some ways
technically superior, formulation of a family of sets, called \emph{fibrations}.  According to this view the family of
sets is identified with a \emph{fibration} $p:Y\to X$, where $Y$, the \emph{total space of $p$}, is the amalgamation of
all of the $Y_x$'s, and $p$, the \emph{display map}, identifies, for each $y\in Y$ the unique $x\in X$ such that
$p\parens{y}=x$.  The element $y\in Y$ is said to \emph{lie over} the index $p(y)\in X$.  The two views of families are
equivalent in that each can be recovered from the other.  Given $p:Y\to X$ we may define a family of sets
$\famset{\fib{p}_x}{x\in X}$ as $\fib{p}(x)\eqdef p^{-1}(x)$, the preimage of $x$ under $p$; this always exists because
$\SET$ has equalizers, and so every map in $\SET$ is a fibration.  Conversely, a family of sets $Y=\famset{Y_x}{x\in X}$
determines a fibration $p_Y:(\coprod_{x\in X} Y_x) \to X$ given by the first projection; the fibers of $p_Y$ are
isomorphic to the given $Y_i$'s.

Cubical sets may themselves be thought of as families of sets, indexed by the cube category, $\CC$, rather than another
set.  This sets up two ways for formulate cubical sets, analogous to those just considered for plain sets:
\begin{enumerate}
\item As a \emph{covariant presheaf} (i.e., a \emph{co-presheaf}), a functor $\cs{X}:\CC\to\SET$, that sends dimension $I$ to sets of $I$-cubes, and
  sends cubical morphisms $f:I\to I'$ to functions $\cs{X}_f:\cs{X}_I\to\cs{X}_{I'}$, preserving identities and
  composition.  (This is the definition given in Section~\ref{sec:cs}.)
\item As a \emph{discrete Grothendieck opfibration}, a functor $\cfib{p}_{\cs{X}}:\cs{X}\to\CC$, that sends each object
  $\alpha$ of $\cs{X}$ to its dimension $\cfib{p}_{\cs{X}}(\alpha)$, an object of $\CC$, and each morphism
  $\phi:\alpha\to\alpha':\cs{X}$ to a cubical morphism $\cfib{p}_{\cs{X}}(\phi):
  \cfib{p}_{\cs{X}}(\alpha)\to\cfib{p}_{\cs{X}}(\alpha')$.  Moreover, the functor $\cfib{p}_{\cs{X}}$ determines a cubical
  set in the sense of being a covariant presheaf, $\fib{\cfib{p}_{\cs{X}}}:\CC\to\SET$, as follows:
  \begin{enumerate}
  \item For each object $I:\CC$, the set of objects, the \emph{fiber of $\cfib{p}_{\cs{X}}$ over $I$}, is defined to be
    the set $\fib{\cfib{p}_{\cs{X}}}_I\eqdef\thesetst{\cs{x}\in\cs{X}}{\cfib{p}_{\cs{X}}(\cs{x})=I}$.
  \item For each cubical morphism $f:I\to I':\CC$, there must be
    a function $\fib{\cfib{p}_{\cs{X}}}_f:\fib{\cfib{p}_{\cs{X}}}_I\to\fib{\cfib{p}_{\cs{X}}}_{I'}$ between the
    fibers of $\cfib{p}_{\cs{X}}$.  The choice of lifting must be functorial in $f$.
  \end{enumerate}
  The critical point is that every morphism $f:I\to I'$ have a unique lifting mapping elements of the fiber over $I$ to
  elements of the fiber over $I'$.
\end{enumerate}
These two formulations are equivalent in that each can be constructed from the other.
\begin{enumerate}
\item Given $\cs{X}:\CC\to\SET$ one may form the (discrete) Grothendieck construction to obtain a functor
  $\cfib{p}_{\cs{X}}:\coelt{}{\cs{X}}\to\CC$ from the category of co-elements of $\cs{X}$ to the cube category that is in
  fact a discrete Grothendieck fibration.
\item Given a discrete Grothendieck fibration $\cfib{p}:\cs{X}\to\CC$, the lifting requirement states exactly that
  $\cfib{p}$ determine a cubical set consisting of the fibers of $\cfib{p}$ and functions between them.
\end{enumerate}

There are, as in the preceding examples, two ways to formulate the concept of a family of \emph{cubical sets} $\cs{Y}$ indexed by a
\emph{cubical set} $\cs{X}$.  It is worthwhile to take a moment to consider what this should mean.  Informally, at each
dimension $I$, there is associated to each $I$-cube $x$ in $\cs{X}$ a set of $I$-cubes $\cs{Y}_{x}$.  To be more
precise, a pointwise representation of a cubical family of sets is a functor $\cs{Y}:\coelt{}{\cs{X}}\to\SET$ that
directly specifies the $I$-cubes of $\cs{Y}$ for each $I$-cube $x$ of $\cs{X}$, and specifies how to lift morphisms
$f:(I,x)\to(I',\cs{Y}_f(x))$ to functions $\cs{Y}_I(x)\to\cs{Y}_{I'}(\cs{X}_f(x))$.  It is this formulation of cubical
family of sets that is used in \cite{bch} to represent families of types.
A fibered representation of a cubical family of sets is a cubical set $\extctx{\cs{X}}{\cs{Y}}$ together with a morphism
of cubical sets $\cfib{p}_{\cs{Y}}:\extctx{\cs{X}}{\cs{Y}}\to\cs{X}$ from the \emph{total space} to the \emph{base space}
of the \emph{fibration} $\cfib{p}_{\cs{Y}}$ (or just $\cfib{p}$ for short).  Thus, $\cfib{p}$ is a family of functions
$\cfib{p}_I:\parens{\extctx{\cs{X}}{\cs{Y}}}_I\to\cs{X}_I$ determining, for each $I$-cube $y$ of the total space, the
$I$-cube of the base space over which it lies.  Naturality means that if $f:I\to I'$ is a cubical morphism, 
\begin{align}
  \label{eq:nat-fib}
  \cs{X}_f(\cfib{p}_I(y)) & = \cfib{p}_{I'}(\cs{Y}_f(y)) && (y\in\cs{Y}_I).
\end{align}
A cubical family of elements of $\cfib{p}:\extctx{\cs{X}}{\cs{Y}}\to\cs{X}$ is a \emph{section} of $\cfib{p}$, which is a
morphism $\cs{y}:\cs{X}\to\extctx{\cs{X}}{\cs{Y}}$ of cubical sets such that $\ocomp{\cs{y}}{\cfib{p}}=\id$.  This means
not only that $\cfib{p}_I(\cs{y}_I(x))=x$, but that the naturality condition holds as well:
\begin{displaymath}
  \parens{\extctx{\cs{X}}{\cs{Y}}}_I(\cs{y}_I(x)) = \cs{y}_{I'}(\cs{X}_f(x)).
\end{displaymath}
 
Each fibration determines a pointwise cubical family of sets, $\fib{\cfib{p}}:\coelt{}{\cs{X}}\to\SET$, called the
\emph{fibers of $\cfib{p}$}.  It is defined by the following equations:
\begin{align}
  \label{eq:fib-ptw}
  \fib{\cfib{p}}_I\parens{x} & \eqdef \fib{\cfib{p}_I}(x) = \thesetst{y\in\cs{Y}_I}{\cfib{p}_I(y)=x} \subseteq \cs{Y}_I \\
  \label{eq:fib-mor}
  \fib{\cfib{p}}_f(y) & \eqdef \cs{Y}_f(y) && (f:I\to I')
\end{align}
Naturality of $\cfib{p}$ ensures that $\fib{\cfib{p}}_f : \fib{\cfib{p}}_I\to\fib{\cfib{p}}_{I'}$.  Conversely, each pointwise
cubical of family of sets $\cs{Y}:\coelt{}{\cs{X}}\to\SET$ determines a fibration:
\begin{align}
  \label{eq:ptw-fib}
  \parens{\extctx{\cs{X}}{\cs{Y}}}_I & \eqdef \thesetst{\extsub{x}{y}}{x\in\cs{X}_I\,\ y\in\cs{Y}_I(x)} \\
  \parens{\extctx{\cs{X}}{\cs{Y}}}_f(\extsub{x}{y}) & \eqdef \extsub{\cs{X}_f(x)}{\cs{Y}_f(y)} \\
  \cfib{p}_{\cs{Y}}(\extsub{x}{y}) & \eqdef x
\end{align}
Thus, the two formulations are essentially equivalent.

\smallskip

When considering cubes as identifications, it is natural to demand that identified indices determine equivalent fibers.
At the very least this means that a fibration of cubical sets should determine a \emph{transport function} between the
fibers over identifications in the base space.  If $\cs{\kappa}$ is a $\ext{J}{y}$-cube in $\cs{X}$, then
$\cfib{p}:\cs{Y}\to\cs{X}$ should determine a function%
\footnote{The use of ``$J$'' here suggests that one way to implement $\trans{\cfib{p}}{J}{y}{-}$ through the Kan filler
  is to treat $\cs{\kappa}$ as a box with only one face,
  where all dimensions other than the filling dimension are exactly the extra dimensions $J$.}
\begin{displaymath}
  \trans{\cfib{p}}{J}{y}{\cs{\kappa}} : \fib{\cfib{p}}_J\parens{\cs{\kappa}_0}\to\fib{\cfib{p}}_J\parens{\cs{\kappa}_1}.
\end{displaymath}
where $\cs{\kappa}_0=\cs{X}_{\spec{y}{0}}\parens{\cs{\kappa}}$ and $\cs{\kappa}_1=\cs{X}_{\spec{y}{1}}\parens{\cs{\kappa}}$.  Thus,
viewing $\cs{\kappa}$ as an identification between its two faces $\cs{\kappa}_0$ and $\cs{\kappa}_1$ induces an
equivalence between the cubical sets assigned to its left- and right faces.  Moreover, the transport map should be a
homotopy equivalence between the fibers.\footnote{See~\cite{hott-book} for a full discussion of equivalence of types.}

The required transport map may be derived from a generalization of the uniform Kan condition on cubical sets to the
\emph{standard Kan condition on fibrations}.  It is a generalization in that a cubical set $\cs{X}$ satisfies the UKC
iff the unique map $\cs{!}:\cs{X}\to\cs{1}$ is a standard Kan fibration.
The main idea of the Kan condition for fibrations is simply that
the homotopy lifting property is required to lift geometric boxes in $\cs{Y}$ over a index cube in $\cs{X}$ to a
geometric cube over the same index cube.  This is enough to derive the required transport property between fibers whose
indices are identified by some cube.  Just as before, the standard Kan condition on fibrations extends to the
\emph{uniform Kan condition on fibrations} given by~\citeauthor{bch}.  A richer class of boxes, with omitted extra
dimensions, are permitted, and the fillings are required to be natural in the extra dimensions.  Finally, the
equivalence between the geometric and algebraic Kan conditions for cubical sets is extended to fibrations of cubical
sets.

\smallskip

The standard formulation of the Kan condition for fibrations is geometric.  For a morphism $\cfib{p}:\cs{Y}\to\cs{X}$ of
cubical sets to be a Kan fibration, it is enough to satisfy the following \emph{fiberwise lifting property}.
Suppose that $\beta:\fspbox{I}{\emp}{y}\to\cs{Y}$ is a positive geometric box in $\cs{Y}$ and
$\kappa:\fscube{\ext{I}{y}}\to\cs{X}$ is a geometric cube in $\cs{X}$.
Let $\iota_0:\fspbox{I}{\emp}{y}\hookrightarrow\fscube{\ext{I}{y}}$ be the inclusion of the free-standing box into
the free-standing cube.  We say the box $\beta$ in $\cs{Y}$ \emph{lies over} the geometric cube $\kappa$ in $\cs{X}$ if
this diagram commutes:
\begin{displaymath}
  \begin{tikzcd}
    \fspbox{I}{\emp}{y} \arrow{r}{\beta} \arrow[hook]{d}{\iota_0} & \cs{Y} \arrow{d}{\cfib{p}} \\
    \fscube{\ext{I}{y}} \arrow{r}{\kappa} & \cs{X}
  \end{tikzcd}
\end{displaymath}
The standard positive geometric Kan condition for $\cfib{p}$ states that there is a
geometric cube $\plift{}{I}{\emp}{y}\parens{\kappa;\beta}:\fscube{\ext{I}{y}}\to\cs{Y}$ such that
\begin{enumerate}
  \item The cube $\plift{}{I}{\emp}{y}\parens{\kappa;\beta}$ is a lifting of the box $\beta$ over the index cube $\kappa$ in that it restricts to $\beta$ along the
    inclusion: $\ocomp{\iota_0}{\plift{}{I}{\emp}{y}\parens{\kappa;\beta}}=\beta$.
  \item The cube $\plift{}{I}{\emp}{y}\parens{\kappa;\beta}$ lies over the given cube $\kappa$:
    $\ocomp{\plift{}{I}{\emp}{y}\parens{\kappa;\beta}}{\cfib{p}}=\kappa$.
\end{enumerate}
(Of course, the same requirements are imposed on negative boxes as well.)  See Figure~\ref{fig:kan-fib} for a
diagrammatic formulation of the standard Kan conditions on fibrations.  In the standard case no ``extra'' dimensions are
permitted in a box, so the parameter $J$ on the boxes involved is always $\emp$.

\begin{figure}[ht]
  \centering
  \begin{tikzcd}
    \fspbox{I}{\emp}{y} \arrow{r}{\beta} \arrow[hook]{d}{\iota_0} & \cs{Y} \arrow{d}{\cfib{p}} \\
    \fscube{\ext{I}{y}} \arrow[dashed]{ur} \arrow{r}{\kappa} & \cs{X}
  \end{tikzcd}
  \qquad
  \begin{tikzcd}
    \fsnbox{I}{\emp}{y} \arrow{r}{\beta} \arrow[hook]{d}{\iota_1} & \cs{Y} \arrow{d}{\cfib{p}} \\
    \fscube{\ext{I}{y}} \arrow[dashed]{ur} \arrow{r}{\kappa} & \cs{X}
  \end{tikzcd}
  \caption{Standard Geometric Kan Conditions for Fibrations (Liftings Dashed)}
  \label{fig:kan-fib}
\end{figure}
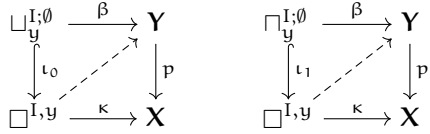

The transport mapping between the fibers of a geometric Kan fibration are readily derived from the geometric Kan
condition.  The identification, $\kappa$, in the base space may be seen as a morphism
$\kappa:\fscube{\fset{y}}\to\cs{X}$ that draws a line in $\cs{X}$.  The transport map may be oriented in either
direction; let us consider here the positive orientation.  There is an inclusion
$\iota_0:\fspbox{\emp}{\emp}{y}\hookrightarrow\fscube{\fset{y}}$ of the positive free-standing box with filling
direction $y$ that specifies the left end-point of the free-standing line as the starting face.  The starting point of
the transport, $y_0\in\fib{\cfib{p}}_{\emp}\parens{\cs{\kappa}_0}$, determines a box
$\beta_0:\fspbox{\emp}{\emp}{y}\to\cs{Y}$ as the geometric realization of the one-point complex $\fset{y_0}$ in
$\cs{Y}_\emp$.  The lifting property determines a geometric line $\lambda=\plift{}{\emp}{\emp}{y}(\beta_0)$ in $\cs{Y}$
such that $\cs{\lambda}_{\spec{y}{0}}=y_0$ and $\cs{\lambda}_{\spec{y}{1}}$ is its destination, determined by $\kappa$,
in the $\emp$-fiber of $\cfib{p}$ over $\cs{X}_{\spec{x}{1}}\parens{\kappa}$.  That is,
\begin{align}
  \label{eq:trans-from-lift}
  \trans{\cfib{p}}{\emp}{y}{\kappa}(y_0) & = \plift{}{\emp}{\emp}{y}\parens{\prealize{}{\emp}{\emp}{y}\parens{\fset{y_0}}}_{\emp}\parens{\spec{y}{1}} \\
  & = \cs{Y}_{\spec{y}{1}}\parens{\pfill{}{\emp}{\emp}{y}\parens{\fset{y_0}}}.
\end{align}

\begin{figure}[tp]
  \centering
  \begin{tikzcd}
    \fspbox{I}{J'}{y}
    \arrow{rr}{\parens{\fspbox{I}{-}{y}}_h}
    \arrow[hook]{dd}[swap]{\iota_0}
    \arrow[bend left=40]{rrr}{\beta'}
    &&
    \fspbox{I}{J}{y} \arrow{r}[near start]{\beta}
    \arrow[hook]{dd}[swap, near end]{\iota_0}
    &
    \cs{Y}
    \arrow{dd}{\cfib{p}}
    \\
    \\
    \fscube{\bdcomb{I}{J'}{y}}
    \arrow{rr}[swap]{\parens{\fscube{\bdcomb{I}{-}{y}}}_h}
    \arrow[dashed]{rrruu}
    &&
    \fscube{\bdcomb{I}{J}{y}}
    \arrow[dashrightarrow]{uur}
    \arrow{r}[swap]{\kappa}
    &
    \cs{X}
  \end{tikzcd}
  \qquad
  \begin{tikzcd}
    \fsnbox{I}{J'}{y}
    \arrow{rr}{\parens{\fsnbox{I}{-}{y}}_h}
    \arrow[hook]{dd}[swap]{\iota_1}
    \arrow[bend left=40]{rrr}{\beta'}
    &&
    \fsnbox{I}{J}{y} \arrow{r}[near start]{\beta}
    \arrow[hook]{dd}[swap, near end]{\iota_1}
    &
    \cs{Y}
    \arrow{dd}{\cfib{p}}
    \\
    \\
    \fscube{\bdcomb{I}{J'}{y}}
    \arrow{rr}[swap]{\parens{\fscube{\bdcomb{I}{-}{y}}}_h}
    \arrow[dashed]{rrruu}
    &&
    \fscube{\bdcomb{I}{J}{y}}
    \arrow[dashrightarrow]{uur}
    \arrow{r}[swap]{\kappa}
    &
    \cs{X}
  \end{tikzcd}
  \caption{Uniform Geometric Kan Conditions for Fibrations (Liftings Dashed)}
  \label{fig:uni-gkc-fib}
\end{figure}
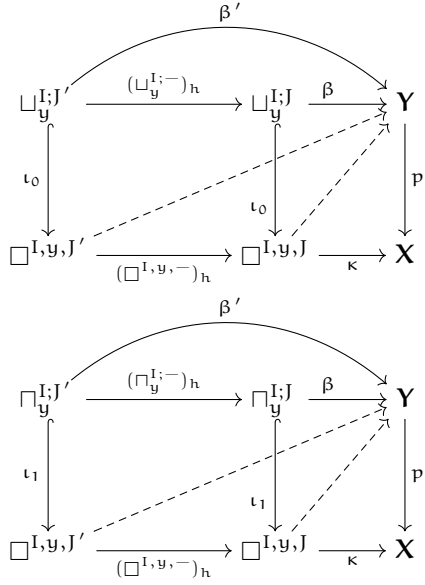

The standard Kan condition for fibrations generalizes to the geometric uniform Kan condition for fibrations by admitting boxes
that have additional dimensions whose opposing faces are omitted (and hence provided by a filling cube).  The required
naturality conditions are illustrated in Figure~\ref{fig:uni-gkc-fib} in their geometric formulation for both positive
and negative boxes.

\smallskip

One can define cubes, boxes, and box projections for a fibration $\cfib{p}:\cs{Y}\to\cs{X}$
and develop geometric and algebraic representations as in Section~\ref{sec:ukc}.
A cube or a box in the fibration $\cfib{p}$ is naturally a cube or a box in $\cs{Y}$
lying over some index cube in $\cs{X}$.
In case of cubes, the index cube in $\cs{X}$ can be recovered from the geometric cube in $\cs{Y}$
by post-composition along with $p$;
therefore a geometric cube in $\cfib{p}$ can simply be a geometric cube in $\cs{Y}$.
A positive geometric box in $\cfib{p}$, written $\tuple{\kappa}{\beta}$,
consists of a geometric box $\beta$ in $\cs{Y}$ and the index cube $\kappa$ in $\cs{X}$
such that $\ocomp{\beta}{\cfib{p}} = \ocomp{\iota_0}{\kappa}$,
which is to say $\beta$ lies over $\kappa$.
\begin{displaymath}
  \begin{tikzcd}
    \fspbox{I}{J}{y} \arrow{r}{\beta} \arrow[hook]{d}{\iota_0} & \cs{Y} \arrow{d}{\cfib{p}} \\
    \fscube{\bdcomb{I}{J}{y}} \arrow{r}{\kappa} & \cs{X}
  \end{tikzcd}
\end{displaymath}
In other words, the set of positive geometric boxes in the fibration $\cfib{p}$,
written $\geopbox{p}{I}{J}{y}$, is the pullback of the cospan
\[
  \cospandiagram{\hom{}{\fscube{\bdcomb{I}{J}{y}}}{\cs{X}}}
  {\hom{}{\fspbox{I}{J}{y}}{\cs{Y}}}
  {\hom{}{\fspbox{I}{J}{y}}{\cs{X}}}
  {\ocomp{\iota_0}{-}}{\ocomp{-}{\cfib{p}}}.
\]
The set of negative geometric boxes, written $\geonbox{p}{I}{J}{y}$, is defined analogously.
The positive geometric box projection sends a geometric cube $\kappa:\fscube{\bdcomb{I}{J}{y}}\to\cs{Y}$
to $\tuple{\ocomp{\kappa}{\cfib{p}}}{\ocomp{\iota_0}{\kappa}}$ as a positive geometric box.
This gives well-defined positive boxes because the box $\ocomp{\iota_0}{\kappa}$ must
lie over the cube $\ocomp{\kappa}{\cfib{p}}$.
The negative projection is defined analogously.
A lifting operation, just as before, is a section of a box projection
and the uniformity means naturality in $J$.

A more algebraic, combinatorial description may be obtained as well.
Let the fibration in question again be $\cfib{p}:\cs{Y}\to\cs{X}$.
The collection of algebraic $I$-cubes in the fibration $\cfib{p}$ is just $\cs{Y}_I$
because geometric cubes in the fibration $\cfib{p}$ are simply geometric cubes in $\cs{Y}$.
A positive algebraic box in the fibration $\cfib{p}$ is then a pair
$\tuple{\alg{\kappa}}{\alg{\beta}}$ of an algebraic box $\alg{\beta}$ in $\algpbox{Y}{I}{J}{y}$
and an algebraic (index) cube $\alg{\kappa}$ in $\cs{X}_{\bdcomb{I}{J}{y}}$
such that $\alg{\beta}$ lies over $\alg{\kappa}$,
or equivalently, every face of $\alg{\beta}$ lies over the matching aspect of $\alg{\kappa}$.
Formally speaking, it requires that for any face map $f\in\pafm{I}{y}$,
\begin{equation}
  \label{eq:alg-over}
  \cfib{p}_{\cod{f}}\parens{\alg{\beta}_f} = \cs{X}_f\parens{\alg{\kappa}}.
\end{equation}
The collection of such positive boxes is written as $\algpbox{p}{I}{J}{y}$,
and the negative counterpart is $\algnbox{p}{I}{J}{y}$.
Finally, the algebraic box projection is the combination
of the fibrations $\cfib{p}$ (for the index cube) and the box projection in $\cs{Y}$ as follows.
\begin{gather*}
  \paproj{I}{J}{y}:\cs{Y}_{\bdcomb{I}{J}{y}}\to\algpbox{\cfib{p}}{I}{J}{y}
  \\
  \paproj{I}{J}{y}\parens{\alg{\kappa}} \eqdef{}
  \tuple{\cfib{p}_{\bdcomb{I}{J}{y}}\parens{\alg{\kappa}}}{\fset{\cs{Y}_f\parens{\alg{\kappa}}}_{f\in\pafm{I}{y}}}
  \\
  \naproj{I}{J}{y}:\cs{Y}_{\bdcomb{I}{J}{y}}\to\algnbox{\cfib{p}}{I}{J}{y}
  \\
  \naproj{I}{J}{y}\parens{\alg{\kappa}} \eqdef{}
  \tuple{\cfib{p}_{\bdcomb{I}{J}{y}}\parens{\alg{\kappa}}}{\fset{\cs{Y}_f\parens{\alg{\kappa}}}_{f\in\nafm{I}{y}}}
\end{gather*}

\begin{proposition}
  $\paproj{I}{J}{y}$ and $\naproj{I}{J}{y}$ for fibrations give well-defined algebraic boxes in $\cfib{p}$
  in the sense that Equation~\eqref{eq:alg-over} holds.
\end{proposition}

\begin{proof}
  For any $f\in\pafm{I}{y}$,
  \begin{align*}
    \cfib{p}_{\cod{f}}\parens{\paproj{I}{J}{y}\parens{\alg{\kappa}}_f}
    &=
    \cfib{p}_{\cod{f}}\parens{\cs{Y}_f\parens{\alg{\kappa}}}
    &&\text{(by definition)}
    \\&=
    \cs{X}_f\parens{\cfib{p}_{\bdcomb{I}{J}{y}}\parens{\alg{\kappa}}}.
    &&\text{(by naturality of $\cfib{p}$)}
  \end{align*}
  The same argument works for negative boxes.
\end{proof}

The geometric and algebraic descriptions enjoy the same equivalence in Section~\ref{sec:ukc},
which is to say that there are natural bijections or identities for cubes, boxes, box projections and uniform filling operations
between the two presentations, as shown in Figure~\ref{fig:fib-gkc-akc}.

\begin{figure}[ht]
  \centering
  \begin{tikzcd}[
      bijection/.style={Leftrightarrow}
    ]
    &
    \algpbox{p}{I}{J}{y}
    \arrow{rr}{\parens{\algpbox{p}{I}{-}{y}}_h}
    \arrow[dashed, bend left]{dd}[near start]{\pfill{}{I}{J}{y}}
    \arrow[leftarrow]{dd}[swap, near start]{\paproj{I}{J}{y}}
    & &
    \algpbox{p}{I}{J'}{y}
    \arrow[dashed, bend left]{dd}[near start]{\pfill{}{I}{J'}{y}}
    \arrow[leftarrow]{dd}[near end, swap]{\paproj{I}{J'}{y}}
    \\
    \geopbox{p}{I}{J}{y}
    \arrow[bijection]{ur}{\tuple{\coyoneda{}}{\hhp}}
    \arrow[crossing over]{rr}[pos=0.2, swap, yshift=-0.3em]{\parens{\geopbox{p}{I}{-}{y}}_h}
    \arrow[dashed, bend left]{dd}[pos=0.45]{\plift{}{I}{J}{y}}
    \arrow[leftarrow]{dd}[near end, swap]{\tuple{\ocomp{-}{\cfib{p}}}{\ocomp{\iota_0}{-}}}
    & &
    \geopbox{p}{I}{J'}{y}
    \arrow[bijection]{ur}{\tuple{\coyoneda{}}{\hhp}}
    \\
    &
    \cs{Y}_{\bdcomb{I}{J}{y}}
    \arrow{rr}[near start]{\parens{\cs{Y}_{\parens{\bdcomb{I}{-}{y}}}}_h}
    & &
    \cs{Y}_{\bdcomb{I}{J'}{y}}
    \\
    \hom{\CSET}{\fscube{\bdcomb{I}{J}{y}}}{\cs{Y}}
    \arrow[bijection]{ur}{\coyoneda{}}
    \arrow{rr}{\hom{\CSET}{\fscube{\bdcomb{I}{-}{y}}}{\cs{Y}}_h}
    & &
    \hom{\CSET}{\fscube{\bdcomb{I}{J'}{y}}}{\cs{Y}}
    \arrow[bijection]{ur}{\coyoneda{}}
    \arrow[leftarrow, dashed, crossing over, bend right]{uu}[swap, near end]{\plift{X}{I}{J}{y}}
    \arrow[crossing over]{uu}[near start]{\tuple{\ocomp{-}{\cfib{p}}}{\ocomp{\iota_0}{-}}}
  \end{tikzcd}

  \bigskip

  \begin{tabular}{lllll}
    \textit{Elements} & \textit{Geometric} & \textit{Algebraic} & \textit{Theorem}
    \\
    \midrule
    Cubes & $\fscube{I} \to \cs{Y}$       & $\cs{Y}_I$             & Natural Bijection
    \\
    Positive Boxes
    & $\geopbox{p}{I}{J}{y}$
    & $\algpbox{p}{I}{J}{y}$ & Natural Bijection
    \\
    Negative Boxes
    & $\geonbox{p}{I}{J}{y}$
    & $\algnbox{p}{I}{J}{y}$ & Natural Bijection
    \\
    Pos.\ Box Projection & $\tuple{\ocomp{-}{\cfib{p}}}{\ocomp{\iota_0}{-}}$ & $\paproj{I}{J}{y}$ & Identity
    \\
    Neg.\ Box Projection & $\tuple{\ocomp{-}{\cfib{p}}}{\ocomp{\iota_1}{-}}$ & $\naproj{I}{J}{y}$ & Identity
    \\
    Pos.\ Uniform Filling & $\plift{X}{I}{J}{y}$ & $\pfill{X}{I}{J}{y}$ & Type Equivalence
    \\
    Neg.\ Uniform Filling & $\nlift{X}{I}{J}{y}$ & $\nfill{X}{I}{J}{y}$ & Type Equivalence
  \end{tabular}
  \caption{Algebraic and Geometric Kan Structures for a Fibration $\cfib{p}:\cs{Y}\to\cs{X}$}
  \label{fig:fib-gkc-akc}
\end{figure}

\section{Discussion and Conclusion}
\label{sec:conclusion}

The uniform Kan condition is a central notion in the model of higher-dimensional type theory in the category of cubical
sets given by~\cite{bch}.  Inspired by their work, we relate a geometric formulation of the UKC, which is expressed in
terms of fibrations and lifting properties, to an algebraic formulation that is closer to the one given by~\cite{bch},
but which also provides another characterization of open boxes.  Geometric and algebraic open boxes are related by a
Yoneda-like correspondence between the morphisms from co-sieves and the algebraic definition given here.  The uniformity
condition on box-filling given by~\cite{bch} may be seen via the geometric characterization as naturality in the extra
dimensions of an open box.

\cite{bch} avoid using Kan fibrations in the form considered here to model families of types, because to do so would
incur the well-known coherence problems arising from modeling exact conditions on substitution in type theory by
universal conditions that do not meet these exactness requirements.  On the other hand, several
authors~\citep{simplicial,awodey:naturalmodel,hofmann:interpretation,coherencecomprehension,curien+:catdeptypes} have
proposed ways to overcome the coherence problems by constructing refined fibration-based models that validate the
required exactness properties of type theory.  It is conceivable that one can build a model of type theory in terms of
uniform Kan fibrations by applying these ideas.

\section{Acknowledgements}
The foregoing presentation draws heavily on Williamson's survey of combinatorial homotopy theory from a cubical
perspective~\citep{williamson2012combinatorial} and, of course, on \cite{bch}, which inspired it.  Numerous
conversations with Carlo Angiuli, Steve Awodey, Spencer Breiner, Dan Licata, Ed Morehouse and Vladimir Voevodsky were
instrumental in preparing this article.

\nocite{breiner:cubical-sets}
\nocite{constable-et-al1985}
\nocite{breiner14}

\bibliographystyle{abbrvnat}
\bibliography{csm}

\end{document}